\documentclass{amsart}

\usepackage{amsmath,amsthm,amsfonts,amssymb,amscd,latexsym}

\newtheorem{thm}{Theorem}[section]
\newtheorem{lem}[thm]{Lemma}
\newtheorem{pro}[thm]{Proposition}
\newtheorem{cor}[thm]{Corollary}

\numberwithin{equation}{section}

\newcommand{\ch}{\mathrm{char}}
\newcommand{\chelem}{\mathrm{ch}}
\newcommand{\der}{\mathrm{Der}}

\newcommand{\ad}{\mathrm{ad}}
\newcommand{\id}{\mathrm{id}}
\newcommand{\Hom}{\mathrm{Hom}}
\newcommand{\End}{\mathrm{End}}
\newcommand{\Ker}{\mathrm{Ker}}
\newcommand{\im}{\mathrm{Im}}
\newcommand{\coker}{\mathrm{Coker}}
\newcommand{\ann}{\mathrm{Ann}}
\newcommand{\rad}{\mathrm{Rad}}
\newcommand{\leib}{\mathrm{Leib}}
\newcommand{\lie}{\mathrm{Lie}}
\newcommand{\sym}{\mathrm{sym}}
\newcommand{\bi}{\mathrm{bi}}
\newcommand{\rel}{\mathrm{rel}}
\newcommand{\CCE}{\mathrm{C}}
\newcommand{\BCE}{\mathrm{B}}
\newcommand{\HCE}{\mathrm{H}}
\newcommand{\CL}{\mathrm{CL}}
\newcommand{\CR}{\mathrm{CR}}
\newcommand{\dl}{\mathrm{d}}
\newcommand{\HR}{\mathrm{HR}}
\newcommand{\HL}{\mathrm{HL}}
\newcommand{\cl}{\mathrm{\widetilde{CL}}}
\newcommand{\dd}{\widetilde{\mathrm{d}}}
\newcommand{\hl}{\mathrm{\widetilde{HL}}}
\newcommand{\GL}{\mathrm{GL}}

\newcommand{\N}{\mathbb{N}}
\newcommand{\F}{\mathbb{F}}
\newcommand{\C}{\mathbb{C}}

\newcommand{\lf}{\mathfrak{L}}
\newcommand{\If}{\mathfrak{I}}
\newcommand{\qf}{\mathfrak{Q}}
\newcommand{\gf}{\mathfrak{g}}
\newcommand{\kf}{\mathfrak{k}}
\newcommand{\gq}{\mathfrak{q}}
\newcommand{\slf}{\mathfrak{sl}}
\newcommand{\ssf}{\mathfrak{s}}
\newcommand{\hf}{\mathfrak{h}}
\newcommand{\nf}{\mathfrak{N}}
\newcommand{\af}{\mathfrak{a}}
\newcommand{\Af}{\mathfrak{A}}
\newcommand{\bbf}{\mathfrak{p}}

\newcommand{\W}{\mathcal{W}}

\begin{document}

%%%%%%%%%%%%%%%%%%%%%%%%%%%%%%%%%%%%%%%%%%%%%%%%%%

\title[Leibniz cohomology]{On Leibniz cohomology}

\author{J\"org Feldvoss}
\address{Department of Mathematics and Statistics, University of South Alabama,
Mobile, AL 36688-0002, USA}
\email{jfeldvoss@southalabama.edu}

\author{Friedrich Wagemann}
\address{Laboratoire de math\'ematiques Jean Leray, UMR 6629 du CNRS, Universit\'e
de Nantes, 2, rue de la Houssini\`ere, F-44322 Nantes Cedex 3, France}
\email{wagemann@math.univ-nantes.fr}



%%%%%%%%%%%%%%%%%%%%%%%%%%%%%%%%%%%%%%%%%%%%%%%%%%

\begin{abstract}
In this paper we prove the Leibniz analogue of Whitehead's vanishing theorem for
the Chevalley-Eilenberg cohomology of Lie algebras. As a consequence, we obtain
the second Whitehead lemma for Leibniz algebras. Moreover, we compute the
cohomology of several Leibniz algebras with adjoint or irreducible coefficients. Our
main tool is a Leibniz analogue of the Hochschild-Serre spectral sequence, which is
an extension of the dual of a spectral sequence of Pirashvili for Leibniz homology
from symmetric bimodules to arbitrary bimodules.
\end{abstract}

%%%%%%%%%%%%%%%%%%%%%%%%%%%%%%%%%%%%%%%%%%%%%%%%%%

\date{November 3, 2020}
          
\maketitle

%%%%%%%%%%%%%%%%%%%%%%%%%%%%%%%%%%%%%%%%%%%%%%%%%%

\section*{Introduction}

%%%%%%%%%%%%%%%%%%%%%%%%%%%%%%%%%%%%%%%%%%%%%%%%%%

In \cite{ALO}, the authors study the cohomology of semi-simple Leibniz algebras,
i.e., the cohomology of finite-dimensional Leibniz algebras $\lf$ with an ideal of
squares $\leib(\lf)$ such that the corresponding canonical Lie algebra $\lf_\lie:=
\lf/\leib(\lf)$ is semi-simple, and conjecture that $\HL^2(\lf,\lf_\ad)=0$. In \cite{FMM}, 
the authors determine the deviation of the second Leibniz cohomology of a complex
Lie algebra with adjoint or trivial coefficients from the corresponding Chevalley-Eilenberg
cohomology. With these motivations in mind, we systematically transpose Pirashvili's
results and tools from homology (see \cite{P}) to cohomology, generalize the dual
of one of Pirashvili's spectral sequences from symmetric bimodules to arbitrary
bimodules, and prove the conjecture mentioned above.

Obtaining this kind of vanishing results would be easy with a strong analogue of the
Hochschild-Serre spectral sequence for Leibniz cohomology. Recall that the Hochschild-Serre
spectral sequence for a Lie algebra extension $0\to\kf\to\gf\to\gq\to 0$ arises from
a filtration of the standard cochain complex of $\gf$ by cochains which vanish in case
one inserts for a certain fixed number $q$ elements of the ideal $\kf$ in $q$ arguments
of the cochain (see \cite[Sections 2 and 3]{HS}). When trying to generalize this filtration
from Lie algebras to Leibniz algebras, one needs to choose whether to filter from the left
or from the right. Another difficulty is that the arising spectral sequence does not converge
to the cohomology of the Leibniz algebra, but rather to the cohomology of some quotient
complex. Furthermore, one must impose that the ideal acts trivially from the left (right) on the
left (right) Leibniz algebra. This last issue excludes the application of the spectral sequence
to many interesting ideals in the Leibniz algebra. Pirashvili \cite[Theorem C]{P} has constructed
an analogue of the Hochschild-Serre spectral sequence using the filtration from the right
for right Leibniz algebras and indicated how to use it together with a long exact sequence
in order to extract cohomology. We use Pirashvili's framework and extend the dual of
his spectral sequence from symmetric bimodules to arbitrary bimodules (see Theorem~\ref{hs1}).
The two main changes of perspective with respect to \cite{P} are the systematic use
of arbitrary bimodules and computations in which we consider ground fields of all
characteristics. We hope that this might be useful for further applications in the future.

The main application of Theorem \ref{hs1} is Theorem \ref{vansemsim} in which we
compute the cohomology of a finite-dimensional semi-simple Leibniz algebra over a
field of characteristic zero with coefficients in an arbitrary finite-dimensional bimodule.
The case $n=2$ of Theorem \ref{vansemsim} is the second Whitehead lemma for
Leibniz algebras. But note that contrary to Chevalley-Eilenberg cohomology, Leibniz
cohomology vanishes in any degree $n\ge 2$. This is one of several instances that
we found by our computations in this paper which indicates that Leibniz cohomology
behaves more uniformly than Chevalley-Eilenberg cohomology. We also show by
examples that the theorem fails in prime characteristic or for infinite-dimensional
modules (see Examples E and F, respectively).

As an immediate consequence of Theorem \ref{vansemsim} we obtain the rigidity
of finite-dimensional semi-simple Leibniz algebras in characteristic zero (see Corollary
\ref{rigid}). More generally, we obtain a complete description of the cohomology
of a finite-dimensional semi-simple left Leibniz algebra with coefficients in the adjoint
bimodule and its (anti-)symmetric counterparts (see Theorem \ref{adj}). In particular,
we deduce that a finite-dimensional semi-simple non-Lie Leibniz algebra in characteristic
zero always possesses outer derivations (see Corollary \ref{outder}) which might
be somewhat surprising as this shows that derivations of non-Lie Leibniz algebras are
more complicated than derivations of Lie algebras.

In addition to the results just mentioned, we dualize another spectral sequence
obtained by Pirashvili for Leibniz homology (see \cite[Theorem A]{P}) that relates
the Leibniz cohomology of a Lie algebra to its Chevalley-Eilenberg cohomology (see
Theorems \ref{theorem_A} and \ref{vanlie}). As an application we generalize some
known results on rigidity to complete Lie algebras (see Corollary \ref{complet} and
Corollary \ref{witt}) and to parabolic subalgebras of finite-dimensional semi-simple
Lie algebras (see Proposition~\ref{borel}). Moreover, we compute the Leibniz
cohomology for the non-abelian two-dimensional Lie algebra (see Example A) and
the three-dimensional Heisenberg algebra (see Example B) with coefficients in
irreducible Leibniz bimodules. The authors believe that Leibniz cohomology is an
important invariant of a Lie algebra that behaves more uniformly than Chevalley-Eilenberg
cohomology. The motivation for including so many details in Section \ref{cel} was
to provide the reader with a solid foundation for computing this invariant in arbitrary
characteristics.

The subject of Leibniz algebras, and especially its (co)homology theory, owes a
great deal to Jean-Louis Loday and Teimuraz Pirashvili (see \cite{CH}, \cite{L},
\cite{LP}, \cite{P}, and \cite{LP2}). Many fundamental definitions and tools are
due to them. While Loday and Pirashvili work with right Leibniz algebras, we work
with left Leibniz algebras. Obviously, results for left Leibniz algebras are equivalent
to the corresponding results for right Leibniz algebras. For the convenience of the
reader we shall indicate where to find the corresponding formulae for left Leibniz
algebras, even when they have been invented in the framework of right Leibniz
algebras and are due to Loday and Pirashvili.

Teimuraz Pirashvili spotted an error in a first version of this article (see \cite{P1}) which
we have subsequently corrected. The error is related to another Hochschild-Serre type
spectral sequence (see Remark (c) after Corollary \ref{pirashvili}) which we have
removed from the present version because its $E_2$-term is more involved than
we originally thought. 

In this paper we will follow the notation used in \cite{F}. All tensor products are over the
relevant ground field and will be denoted by $\otimes$. For a subset $X$ of a vector space
$V$ over a field $\F$ we let $\langle X\rangle_\F$ be the subspace of $V$ spanned by $X$.
We will denote the space of linear transformations from an $\F$-vector space $V$ to an
$\F$-vector space $W$ by $\Hom_\F(V,W)$. In particular, $V^*:=\Hom_\F(V,\F)$ will be
the space of linear forms on a vector space $V$ over a field $\F$. Moreover, $S^2(V)$ will
denote the symmetric square of a vector space $V$. Finally, the identity function on a set
$X$ will be denoted by $\id_X$, and the set $\{0,1,2,\dots\}$ of non-negative integers will
be denoted by $\N_0$.

%%%%%%%%%%%%%%%%%%%%%%%%%%%%%%%%%%%%%%%%%%%%%%%%%%

\section{Preliminaries}\label{prelim}

%%%%%%%%%%%%%%%%%%%%%%%%%%%%%%%%%%%%%%%%%%%%%%%%%%

In this section we recall some definitions and collect several results that will be useful in the
remainder of the paper.

A {\em left Leibniz algebra\/} is an algebra $\lf$ such that every left multiplication operator
$L_x:\lf\to\lf$, $y\mapsto xy$ is a derivation. This is equivalent to the identity
\begin{equation}\label{LLI}
x(yz)=(xy)z+y(xz)
\end{equation}
for all $x,y,z\in\lf$, which in turn is equivalent to the identity
\begin{equation}\label{RLLI}
(xy)z=x(yz)-y(xz)
\end{equation}
for all $x,y,z\in\lf$. We will call both identities the {\em left Leibniz identity\/}. There is a similar
definition of a {\em right Leibniz algebra\/} but  in this paper we will only consider left Leibniz
algebras.

Every left Leibniz algebra has an important ideal, its Leibniz kernel, that measures how much
the Leibniz algebra deviates from being a Lie algebra. Namely, let $\lf$ be a left Leibniz algebra
over a field $\F$. Then $$\leib(\lf):=\langle x^2\mid x\in\lf\rangle_\mathbb{F}$$ is called the
{\em Leibniz kernel\/} of $\lf$. The Leibniz kernel $\leib(\lf)$ is an abelian ideal of $\lf$, and
$\leib(\lf)\ne\lf$ when $\lf\ne 0$ (see \cite[Proposition 2.20]{F}). Moreover, $\lf$ is a Lie
algebra if, and only if, $\leib(\lf)=0$. It follows from the left Leibniz identity (\ref{RLLI})
that $\leib(\lf)\subseteq C_\ell(\lf)$, where $C_\ell(\lf):=\{c\in\lf\mid\forall\,x\in\lf:cx=0\}$
denotes the {\em left center\/} of $\lf$. 

By definition of the Leibniz kernel, $\lf_\lie:=\lf/\leib(\lf)$ is a Lie algebra which we call the
{\em canonical Lie algebra\/} associated to $\lf$. In fact, the Leibniz kernel is the smallest
ideal such that the corresponding factor algebra is a Lie algebra (see \cite[Proposition 2.22]{F}).

Next, we will briefly discuss left modules and bimodules of left Leibniz algebras. Let $\lf$ be a
left Leibniz algebra over a field $\F$. A {\em left $\lf$-module\/} is a vector space $M$ over
$\F$ with an $\F$-bilinear left $\lf$-action $\lf\times M\to M$, $(x,m)\mapsto x\cdot m$ such
that
\begin{equation}\label{LLM}
(xy)\cdot m=x\cdot(y\cdot m)-y\cdot(x\cdot m)
\end{equation}
is satisfied for every $m\in M$ and all $x,y\in\lf$.

By virtue of \cite[Lemma 3.3]{F}, every left $\lf$-module is an $\lf_\lie$-module, and vice
versa. Therefore left Leibniz modules are sometimes called Lie modules. Consequently, many
properties of left Leibniz modules follow from the corresponding properties of modules for
the canonical Lie algebra.

The correct concept of a module for a left Leibniz algebra $\lf$ is the notion of a Leibniz
bimodule. An {\em $\lf$-bimodule\/} is a left $\lf$-module $M$ with an $\F$-bilinear right
$\lf$-action $M\times\lf\to M$, $(m,x)\mapsto m\cdot x$ such that
\begin{equation}\label{LML}
(x\cdot m)\cdot y=x\cdot(m\cdot y)-m\cdot(xy)
\end{equation}
and
\begin{equation}\label{MLL}
(m\cdot x)\cdot y=m\cdot(xy)-x\cdot(m\cdot y)
\end{equation}
are satisfied for every $m\in M$ and all $x,y\in\lf$. In fact, all three identities (\ref{LLM}),
(\ref{LML}), and (\ref{MLL}) are instances of the left Leibniz identity, written down for
the left Leibniz algebra $\lf\oplus M$ which is considered as an abelian extension in the
theory of non-associative algebras, where the element $m$ coccurs on the right, in the
middle, or on the left, respectively (for the details see \cite[Section 3]{F}).  

The usual definitions of the notions of {\em sub-(bi)module\/}, {\em irreducibility\/},
{\em complete reducibility\/}, {\em composition series\/}, {\em homomorphism\/},
{\em isomorphism\/}, etc., hold for left Leibniz modules and Leibniz bimodules.

Let $\lf$ be a left Leibniz algebra over a field $\F$, and let $M$ be an $\lf$-bimodule. Then
$M$ is said  to be {\em symmetric\/} if $m\cdot x=-x\cdot m$ for every $x\in\lf$ and every
$m\in M$, and $M$ is said to be {\em anti-symmetric\/} if $m\cdot x=0$ for every $x\in\lf$
and every $m\in M$. We call $$M_0:=\langle x\cdot m+m\cdot x\mid x\in\lf,m\in M
\rangle_\F$$ the {\em anti-symmetric kernel\/} of $M$. It is known that $M_0$ is an
anti-symmetric $\lf$-sub-bimodule of $M$ (see \cite[Proposition 3.12]{F}) such that
$M_\sym:=M/M_0$ is symmetric (see \cite[Proposition 3.13]{F}).

Moreover, for any $\lf$-bimodule $M$ we will need its {\em space of right $\lf$-invariants\/}
$$M^\lf:=\{m\in M\mid\forall\,x\in\lf:m\cdot x=0\}$$ and the {\em annihilator\/} $$\ann_\lf^\bi
(M):=\{x\in\lf\mid\forall\,m\in M:x\cdot m=0=m\cdot x\}\,.$$

Our first result will be useful in the proof of Theorem \ref{whitehead}.

\begin{lem}\label{sym}
Let $\lf$ be a left Leibniz algebra, and let $M$ be an $\lf$-bimodule such that $M^\lf=0$.
Then $M$ is symmetric. In particular, $\leib(\lf)\subseteq\ann_\lf^\bi(M)$.
\end{lem}

\begin{proof}
Since $M_0$ is anti-symmetric, it follows from the hypothesis that $$M_0=M_0^\lf\subseteq
M^\lf=0\,.$$ Hence we obtain from the definition of $M_0$ that $M$ is symmetric. The second
part is then an immediate consequence of \cite[Lemma 3.10]{F}.
\end{proof}

It is clear from the definition of $M^\lf$ that an $\lf$-bimodule $M$ is anti-symmetric if, and
only if, $M^\lf=M$. We will use Lemma \ref{sym} to show that the symmetry of non-trivial
irreducible Leibniz bimodules can also be characterized by the behavior of their spaces of right
invariants. As a preparation for this, we need to know that the latter space is a sub-bimodule.

\begin{lem}\label{inv}
Let $\lf$ be a left Leibniz algebra, and let $M$ be an $\lf$-bimodule. Then $M^\lf$ is a
sub-bimodule of $M$.
\end{lem}

\begin{proof}
It follows from (\ref{LML}) that $M^\lf$ is invariant under the left action on $M$, and it
follows from (\ref{MLL}) that $M^\lf$ is invariant under the right action on $M$.
\end{proof}
\vspace{-.3cm}

\noindent {\bf Remark.} More generally, the proof of Lemma \ref{inv} shows that $M^\If$ is
an $\lf$-sub-bimodule of $M$ for every left ideal $\If$ of $\lf$.
\vspace{.3cm}

Now we can characterize the symmetry of a non-trivial irreducible Leibniz bimodule by the
vanishing of its space of right invariants. In particular, for non-trivial irreducible Leibniz
bimodules we obtain the converse of Lemma \ref{sym}. (Recall that an irreducible bimodule
$M$ is a bimodule that has exactly two sub-bimodules, namely, $0$ and $M$. In particular,
an irreducible bimodule is by definition a non-zero vector space.)

\begin{cor}\label{irr}
Let $\lf$ be a left Leibniz algebra, and let $M$ be an irreducible $\lf$-bimodule. Then $M$ is
symmetric with non-trivial $\lf$-action if, and only if, $M^\lf=0$.
\end{cor}

\begin{proof}
Since $M$ is irreducible, we obtain from Lemma \ref{inv} that $M^\lf=0$ or $M^\lf=M$.
Suppose first that $M$ is symmetric with non-trivial $\lf$-action. Then we have that
$M^\lf=0$. On the other hand, the converse immediately follows from Lemma \ref{sym}.
\end{proof}
\vspace{-.3cm}

Recall that every left $\lf$-module $M$ of a left Leibniz algebra $\lf$ determines a unique
symmetric $\lf$-bimodule structure on $M$ by defining $m\cdot x:=-x\cdot m$ for every
element $m\in M$ and every element $x\in\lf$ (see \cite[Proposition 3.15\,(b)]{F}). We
will denote this symmetric $\lf$-bimodule by $M_s$. Similarly, every left $\lf$-module $M$
with trivial right action is an anti-symmetric $\lf$-bimodule (see \cite[Proposition 3.15\,(a)]{F}).
We will denote this module by $M_a$. Note that for any irreducible left $\lf$-module $M$
the $\lf$-bimodules $M_s$ and $M_a$ are irreducible, and every irreducible $\lf$-bimodule
arises in this way from an irreducible left $\lf$-module (see \cite[p.\ 415]{LP2}).

Similar to the boundary map in \cite{CH} for the homology of a right Leibniz algebra with
coefficients in a right module one can also introduce a coboundary map $\dd^\bullet$
for the cohomology of a left Leibniz algebra with coefficients in a left module as follows.

Let $\lf$ be a left Leibniz algebra over a field $\F$, and let $M$ be a left $\lf$-module.
For any non-negative integer $n$ set $\CL^n(\lf,M):=\Hom_\F(\lf^{\otimes n},M)$ and
consider the linear transformation $\dd^n:\CL^n(\lf,M)\to\CL^{n+1}(\lf,M)$ defined by
\begin{eqnarray*}
(\dd^nf)(x_1,\dots,x_{n+1}) & := & \sum_{i=1}^{n+1}(-1)^{i+1}x_i\cdot f(x_1,
\dots,\hat{x}_i,\dots,x_{n+1})\\
& + & \sum_{1\le i<j\le n+1}(-1)^if(x_1,\dots,\hat{x}_i,\dots,x_ix_j,\dots,x_{n+1})
\end{eqnarray*}
for any $f\in\CL^n(\lf,M)$ and all elements $x_1,\dots,x_{n+1}\in\lf$.  Note that in the
second sum of the coboundary map $\dd^n$ the term $x_ix_j$ appears in the $j$-th
position. (Note that the convention in Loday's book (see \cite[(10.6.2)]{CH}) is for the
homology of right Leibniz algebras, and it is different.) Moreover, here and in the remainder
of the paper we identify the tensor power $\lf^{\otimes n}$ with the corresponding Cartesian
power.

Now let $M$ be an $\lf$-bimodule and for any non-negative integer $n$ consider the linear
transformation $\dl^n:\CL^n(\lf,M)\to\CL^{n+1}(\lf,M)$ defined by
\begin{eqnarray*}
(\dl^nf)(x_1,\dots,x_{n+1}) & := & \sum_{i=1}^n(-1)^{i+1}x_i\cdot f(x_1,\dots,\hat{x}_i,
\dots,x_{n+1})\\
& + & (-1)^{n+1}f(x_1,\dots,x_n)\cdot x_{n+1}\\
& + & \sum_{1\le i<j\le n+1}(-1)^if(x_1,\dots,\hat{x}_i,\dots,x_ix_j,\dots,x_{n+1})
\end{eqnarray*}
for any $f\in\CL^n(\lf,M)$ and all elements $x_1,\dots,x_{n+1}\in\lf$, where as before
the term $x_ix_j$ appears in the $j$-th position.

It is proved in \cite[Lemma 1.3.1]{C} that $\CL^\bullet(\lf,M):=(\CL^n(\lf,M),\dl^n)_{n\in
\N_0}$is a cochain complex, i.e., $\dl^{n+1}\circ\dl^n=0$ for every non-negative integer
$n$. Of course, the original idea of defining Leibniz cohomology as the cohomology of such
a cochain complex for right Leibniz algebras is due to Loday and Pirashvili \cite[Section 1.8]{LP}.
Hence one can define the {\em cohomology of $\lf$ with coefficients in an $\lf$-bimodule\/}
$M$ by $$\HL^n(\lf,M):=\HCE^n(\CL^\bullet(\lf,M)):=\Ker(\dl^n)/\im(\dl^{n-1})$$ for every
non-negative integer $n$. (In this definition we use that $\dl^{-1}:=0$.)

If $M$ is a symmetric $\lf$-bimodule, then we have the identity $\dd^n=\dl^n$ for any
non-negative integer $n$. Namely,  
\begin{eqnarray*}
(\dd^nf)(x_1,\dots,x_{n+1}) & = & \sum_{i=1}^n(-1)^{i+1}x_i\cdot f(x_1,\dots,\hat{x}_i,
\dots,x_{n+1})\\
& + & (-1)^{n+2}x_{n+1}\cdot f(x_1,\dots,x_n)\\
& + & \sum_{1\le i<j\le n+1}(-1)^if(x_1,\dots,\hat{x}_i,\dots,x_ix_j,\dots,x_{n+1})\\
& = & (\dl^nf)(x_1,\dots,x_{n+1})
\end{eqnarray*}
for any $f\in\CL^n(\lf,M)$ and all elements $x_1,\dots,x_{n+1}\in\lf$. In particular, as
mentioned earlier, any left $\lf$-module $M$ can be turned into a symmetric $\lf$-module
$M_s$, and the fact that $\dl^\bullet$ is a coboundary map for $\CL^\bullet(\lf,M_s)$
shows that $\cl^\bullet(\lf,M):=(\CL^n(\lf,M),\dd^n)_{n\in\N_0}$ is a cochain complex,
i.e., $\dd^{n+1}\circ\dd^n=0$ for every non-negative integer $n$. Hence one can define
the {\em cohomology of $\lf$ with coefficients in a left $\lf$-module\/} $M$ by $$\hl^n
(\lf,M):=\HCE^n(\cl^\bullet(\lf,M)):=\Ker(\dd^n)/\im(\dd^{n-1})$$ for every non-negative
integer $n$. (As in the previous definition we set $\dd^{-1}:=0$.)

Now we are ready to state the next result (see \cite[Lemma 2.2]{P} for the analogous
result in Leibniz homology) whose second part generalizes \cite[Corollary~4.4\,(b)]{F}
to arbitrary degrees and which will be crucial in Section \ref{semsim}. (Note that the
second part has already been obtained in \cite[Proposition 1.3.16]{C}). For the convenience
of the reader we include a detailed proof.

\begin{lem}\label{antisym}
Let $\lf$ be a left Leibniz algebra over a field $\F$, and let $M$ be a left $\lf$-module. Then
the following statements hold:
\begin{enumerate}
\item[{\rm (a)}] If $M$ is considered as a symmetric $\lf$-bimodule $M_s$, then $$\HL^n
                          (\lf,M_s)=\hl^n(\lf,M)$$ for every integer $n\ge 0$.
\item[{\rm (b)}] If $M$ is considered as an anti-symmetric $\lf$-bimodule $M_a$, then
                          $$\HL^0(\lf,M_a)=M$$ and $$\HL^n(\lf,M_a)\cong\hl^{n-1}(\lf,\Hom_\F
                          (\lf,M))=\HL^{n-1}(\lf,\Hom_\F(\lf,M)_s)$$ for every integer $n\ge 1$,
                          where $\Hom_\F(\lf,M)$ is a left $\lf$-module via $$(x\cdot f)(y):=x\cdot
                          f(y)-f(xy)$$ for every $f\in\Hom_\F(\lf,M)$ and any elements $x,y\in\lf$.
\end{enumerate}
\end{lem}

\begin{proof}
By virtue of the computation before Lemma \ref{antisym}, we only need to prove part (b).
Note that the first part of (b) is just \cite[Corollary 4.2\,(b)]{F}.

First, we show that $\Hom_\F(\lf,M)$ is a left $\lf$-module via the given action. Let $f\in
\Hom_\F(\lf,M)$ and $x,y,z\in\lf$ be arbitrary. Then we obtain from the defining identity
of a left Leibniz module (\ref{LLM}) and the left Leibniz identity (\ref{RLLI}) that
\begin{eqnarray*}
((xy)\cdot f)(z) & = & (xy)\cdot f(z)-f((xy)z)\\
& = & x\cdot(y\cdot f(z))-y\cdot(x\cdot f(z))-f(x(yz))+f(y(xz))\,,
\end{eqnarray*}
and
\begin{eqnarray*}
(x\cdot(y\cdot f))(z) & = & x\cdot(y\cdot f)(z)-(y\cdot f)(xz)\\
& = & x\cdot(y\cdot f(z))-x\cdot f(yz)-y\cdot f(xz)+f(y(xz))\,,
\end{eqnarray*}
as well as
\begin{eqnarray*}
(y\cdot(x\cdot f))(z) & = & y\cdot(x\cdot f)(z)-(x\cdot f)(yz)\\
& = & y\cdot(x\cdot f(z))-y\cdot f(xz)-x\cdot f(yz)+f(x(yz))\,.
\end{eqnarray*}
Hence $((xy)\cdot f)(z)=(x\cdot(y\cdot f)(z)-(y\cdot(x\cdot f)(z)$ for every $z\in\lf$, or
equivalently, $(xy)\cdot f=x\cdot(y\cdot f)-y\cdot(x\cdot f)$.

Now we will prove the second part of (b). Let $n$ be any positive integer. Consider the
linear transformations $\varphi^n:\CL^n(\lf,M)\to\CL^{n-1}(\lf,\Hom_\F(\lf,M))$ defined
by $\varphi^n(f)(x_1,\dots,x_{n-1})(x):=f(x_1,\dots,x_{n-1},x)$ for any elements $x_1,
\dots,x_{n-1},x\in\lf$ and $\psi^n:\CL^{n-1}(\lf,\Hom_\F(\lf,M))\to\CL^n(\lf,M)$ defined
by $\psi^n(g)(x_1,\dots,x_{n-1},x_n)$ $:=g(x_1,\dots,x_{n-1})(x_n)$ for any elements
$x_1,\dots,x_{n-1},x_n\in\lf$. Then $\varphi^n$ and $\psi^n$ are inverses of each other.

Next, we will show that $\dd^{n-1}\circ\varphi^n=\varphi^{n+1}\circ\dl^n$. Compute
\begin{eqnarray*}
(\dd^{n-1}\circ\varphi^n)(f)(x_1,\dots,x_n)(x) & = & \dd^{n-1}(\varphi^n(f))(x_1,\dots,x_n)(x)\\
& = & \sum_{i=1}^n(-1)^{i+1}(x_i\cdot\varphi^n(f))(x_1,\dots,\hat{x_i},\dots,x_n)(x)\\
& + & \sum_{1\le i<j\le n}(-1)^i\varphi^n(f)(x_1,\dots,\hat{x_i},\dots,x_ix_j,\dots,x_n)(x)\\
& = & \sum_{i=1}^n(-1)^{i+1}x_i\cdot\varphi^n(f)(x_1,\dots,\hat{x_i},\dots,x_n)(x)\\
& - & \sum_{i=1}^n(-1)^{i+1}\varphi^n(f)(x_1,\dots,\hat{x_i},\dots,x_n)(x_ix)\\
& + & \sum_{i\le i<j\le n}(-1)^if(x_1,\dots,\hat{x_i},\dots,x_ix_j,\dots,x_n,x)\\
& = & \sum_{i=1}^n(-1)^{i+1}x_i\cdot f(x_1,\dots,\hat{x_i},\dots,x_n,x)\\
& + & \sum_{i=1}^n(-1)^if(x_1,\dots,\hat{x_i},\dots,x_n,x_ix)\\
& + & \sum_{1\le i<j\le n}(-1)^if(x_1,\dots,\hat{x_i},\dots,x_ix_j,\dots,x_n,x)
\end{eqnarray*}
and
\begin{eqnarray*}
(\varphi^{n+1}\circ\dl^n)(f)(x_1,\dots,x_n)(x) & = & \varphi^{n+1}(\dl^n(f))(x_1,\dots,x_n)(x)\\
& = & \dl^n(f)(x_1,\dots,x_n,x)\\
& = & \sum_{i=1}^n(-1)^{i+1}x_i\cdot f(x_1,\dots,\hat{x_i},\dots,x_n,x)\\
& + & (-1)^{n+1}f(x_1,\dots,x_n)\cdot x\\
& + & \sum_{1\le i<j\le n}(-1)^if(x_1,\dots,\hat{x_i},\dots,x_ix_j,\dots,x_n,x)\\
& + & \sum_{i=1}^n(-1)^if(x_1,\dots,\hat{x_i},\dots,x_n,x_ix)
\end{eqnarray*}
for any elements $x_1,\dots,x_n,x\in\lf$. Since $M$ is anti-symmetric, the second of the
last four summands vanishes, and thus the two compositions are equal. From the identity
$\dd^{n-1}\circ\varphi^n=\varphi^{n+1}\circ\dl^n$ for every integer $n\ge 1$ we obtain
that $$\varphi^n(\Ker(\dl^n))\subseteq\Ker(\dd^{n-1})$$ and $$\varphi^n(\im(\dl^{n-1}))
\subseteq\im(\dd^{n-2})$$ for every integer $n\ge 1$. Hence $\varphi^n$ induces an
isomorphism of vector spaces between $\HL^n(\lf,M)$ and $\hl^{n-1}(\lf,\Hom_\F(\lf,M))$
for every integer $n\ge 1$. In order to see the remainder of the assertion, apply part (a).
\end{proof}

In the special case of the trivial one-dimensional Leibniz bimodule we obtain from Lemma
\ref{antisym} the following result which will be needed in Section \ref{semsim} (see
\cite[Exercise E.10.6.1]{CH} for the analogous result in Leibniz homology).

\begin{cor}\label{coadj}
Let $\lf$ be a left Leibniz algebra over a field $\F$. Then for every integer $n\ge 1$ there
are isomorphisms $$\HL^n(\lf,\F)\cong\hl^{n-1}(\lf,\lf^*)=\HL^{n-1}(\lf,(\lf^*)_s)$$ of
vector spaces, where $\lf^*$ is a left $\lf$-module via $(x\cdot f)(y):=-f(xy)$ for every
linear form $f\in\lf^*$ and any elements $x,y\in\lf$.
\end{cor}

\noindent {\bf Remark.} Note that \cite[Theorem 3.5]{HPL} is an immediate consequence
of the case $n=2$ of Corollary \ref{coadj} and \cite[Corollary 4.4\,(a)]{F}.
\vspace{.2cm}

%%%%%%%%%%%%%%%%%%%%%%%%%%%%%%%%%%%%%%%%%%%%%%%%%%

\section{A relation between Chevalley-Eilenberg cohomology and Leibniz cohomology for
Lie algebras}\label{cel}

%%%%%%%%%%%%%%%%%%%%%%%%%%%%%%%%%%%%%%%%%%%%%%%%%%

Let $\gf$ be a Lie algebra, and let $M$ be a left $\gf$-module that is also viewed
as a symmetric Leibniz $\gf$-bimodule $M_s$. In this section, we will investigate
how the Chevalley-Eilenberg cohomology $\HCE^\bullet(\gf,M)$ and the Leibniz
cohomology $\HL^\bullet(\gf,M_s)$ are related. The tools  set forth in this section
have been developed by Pirashvili, and we follow the analogous treatment for
homology given in \cite{P} very closely.

The Chevalley-Eilenberg cohomology of a Lie algebra $\gf$ with trivial coefficients
is not isomorphic (up to a degree shift) to the Chevalley-Eilenberg cohomology
of $\gf$ with coadjoint coefficients as it is the case for Leibniz cohomology (see
Corollary~\ref{coadj}). Instead these cohomologies are only related by a long
exact sequence (see Proposition \ref{lescoadj}). The cohomology measuring the
deviation from such an isomorphism will appear in a spectral sequence (see Theorem
\ref{theorem_A}) which can be used to relate the Leibniz cohomology of a Lie
algebra to its Chevalley-Eilenberg cohomology (see Proposition \ref{lesrelcoh}).

The exterior product map $m_n:\Lambda^n\gf\otimes\gf\to\Lambda^{n+1}\gf$
given on homogeneous tensors by $x_1\wedge\ldots\wedge x_n\otimes x\mapsto
x_1\wedge\ldots\wedge x_n\wedge x$ induces a monomorphism$$m^\bullet:
\overline{\CCE}^\bullet(\gf,\F)[-1]\hookrightarrow\CCE^\bullet(\gf,\gf^*)\,,$$
where $\overline{\CCE}^\bullet(\gf,\F)$ is the truncated cochain complex
$$\overline{\CCE}^0(\gf,\F):=0\,\mbox{ and }\,\overline{\CCE}^n(\gf,\F):=
\CCE^n(\gf,\F)\,\,\,\mbox{ for every integer}\,\,\,n>0\,.$$ The cochain complex
$\CR^\bullet(\gf)$ is defined by $\CR^\bullet(\gf):=\coker(m^\bullet)[-1]$, and
the corresponding chain complex is defined by $\CR_\bullet(\gf):=\Ker(m_\bullet)[1]$
(see \cite[(1.2.8), p.\ 9]{HA} for the definition of the degree shift and see \cite[Section
1]{P} for the definition of $\CR_\bullet(\gf)$). We will mainly use the cochain
complex in our paper, but the chain complex appears in Theorem \ref{theorem_A}
and its proof. Observe that classes in $\CR^n(\gf)$ are represented by cochains of
degree $n+1$ with values in $\gf^*$, i.e., they have $n+2$ arguments. From the
short exact sequence $$0\to\overline{\CCE}^\bullet(\gf,\F)[-1]\to\CCE^\bullet
(\gf,\gf^*)\to\CR^\bullet(\gf)[1]\to 0$$ of cochain complexes we obtain the following
long exact sequence:

\begin{pro}\label{lescoadj}
For every Lie algebra $\gf$ over a field $\F$ there is a long exact sequence
\begin{eqnarray*}
0 & \to & \HCE^2(\gf,\F)\to\HCE^1(\gf,\gf^*)\to\HR^0(\gf)\\
& \to & \HCE^3(\gf,\F)\to\HCE^2(\gf,\gf^*)\to\HR^1(\gf)\to\cdots
\end{eqnarray*}
and an isomorphism $\HCE^1(\gf,\F)\cong\HCE^0(\gf,\gf^*)$.
\end{pro}

\noindent {\bf Remark.} If we assume that the characteristic of the ground field $\F$
is not 2, then $\HR^0(\gf)\cong[S^2(\gf)^*]^\gf$ is the space of invariant symmetric
bilinear forms on $\gf$ (see \cite[p.\ 403]{P}). As a consequence, we obtain from Proposition
\ref{lescoadj} in the case that $\ch(\F)\ne 2$ the five-term exact sequence $$0\to\HCE^2
(\gf,\F)\to\HCE^1(\gf,\gf^*)\to[S^2(\gf)^*]^\gf\to\HCE^3(\gf,\F)\to\HCE^2(\gf,\gf^*)\,,$$
which generalizes \cite[Proposition 1.3\,(1) \& (3)]{Fa}. Note that the map $[S^2(\gf)^*]^\gf
\to\HCE^3(\gf,\F)$ is the classical Cartan-Koszul map defined by $\omega\mapsto\overline{\omega}
+\BCE^3(\gf,\F)$, where $\overline{\omega}(x\wedge y\wedge z):=\omega(xy,z)$ for any
elements $x,y,z\in\gf$ (see \cite[p.\ 403]{P}).
\vspace{.3cm}

For a Lie algebra $\gf$ and a left $\gf$-module $M$ viewed as a symmetric Leibniz $\gf$-bimodule
$M_s$, we have a natural monomorphism $$\CCE^\bullet(\gf,M)\hookrightarrow\CL^\bullet(\gf,
M_s)\,.$$ The cokernel of this morphism is by definition (up to a shift in the degree) the cochain
complex $\CCE_\rel^\bullet(\gf,M)$: $$\CCE_\rel^\bullet(\gf,M):=\coker(\CCE^\bullet(\gf,M)
\to\CL^\bullet(\gf,M_s))[-2]\,.$$ We therefore have another long exact sequence. (For the
isomorphisms in degrees 0 and 1 see \cite[Corollary 4.2\,(a)]{F} and \cite[Corollary 4.4\,(a)]{F},
respectively.)

\begin{pro}\label{lesrelcoh}
Let $\gf$ be a Lie algebra, and let $M$ be a left $\gf$-module considered as a symmetric Leibniz
$\gf$-bimodule $M_s$. Then there is a long exact sequence
\begin{eqnarray*}
0 & \to & \HCE^2(\gf,M)\to\HL^2(\gf,M_s)\to\HCE_\rel^0(\gf,M)\\
& \to & \HCE^3(\gf,M)\to\HL^3(\gf,M_s)\to\HCE_\rel^1(\gf,M)\to\cdots
\end{eqnarray*}
and isomorphisms
$$\HL^0(\gf,M_s)\cong\HCE^0(\gf,M),\,\,\,\HL^1(\gf,M_s)\cong\HCE^1(\gf,M)\,.$$
\end{pro}

\noindent {\bf Remark.} If we again assume that the characteristic of the ground field $\F$
is not 2, it follows from Theorem \ref{theorem_A} below in conjunction with the remark after
Proposition~\ref{lescoadj} that $\HCE_\rel^0(\gf,\F)\cong\HR^0(\gf)\cong[S^2(\gf)^*]^\gf$
is the space of invariant symmetric bilinear forms on $\gf$. So when $\ch(\F)\ne 2$, we obtain
the five-term exact sequence $$0\to\HCE^2(\gf,\F)\to\HL^2(\gf,\F)\to[S^2(\gf)^*]^\gf
\to\HCE^3(\gf,\F)\to\HL^3(\gf,\F)$$ as a special case of Proposition \ref{lesrelcoh} (cf.\
\cite[Proposition 3.2]{HPL} for fields of characteristic zero). Note that Corollary \ref{coadj}
implies that the second terms of the five-term exact sequences in Proposition \ref{lescoadj}
and in Proposition \ref{lesrelcoh} for $M:=\F$ are isomorphic, but the fifth terms are not
necessarily isomorphic (see the remark after Example A below).
\vspace{.3cm}

Observe that as for $\CR^n(\gf)$, representatives of classes in $\CCE_\rel^n(\gf,M)$ have
$n+2$ arguments.

On the quotient cochain complex $\CCE_\rel^\bullet(\gf,M)$ there is the following filtration
$${\mathcal F}^p\CCE_\rel^{n}(\gf,M)=\{[c]\in\CCE_\rel^{n}(\gf,M)\mid c(x_1,\ldots,
x_{n+2})=0\mbox{ if }\exists\,j\leq p+1:\,x_{j-1}=x_j\}\,.$$ Note that the condition is
independent of the representative $c$ of the class $[c]$. This defines a finite decreasing
filtration
\begin{equation}  \label{finite_filtration_1}
{\mathcal F}^0\CCE_\rel^{n}(\gf,M)=\CCE_\rel^{n}(\gf,M)\supset{\mathcal
F}^1\CCE_\rel^{n}(\gf,M)\supset\cdots\supset{\mathcal F}^{n+1}\CCE_\rel^{n}(\gf,M)
=\{0\}\,.
\end{equation}
Then we have the following result:

\begin{lem}\label{lemma_filtration_compatible}
This filtration is compatible with the Leibniz coboundary map $\dl^\bullet$.
\end{lem}

\begin{proof}
The Leibniz coboundary map, acting on a cochain $c$, is an alternating sum of operators
$\dl_{ij} (c)$, $\delta_i(c)$ and $\partial(c)$, where $\dl_{ij}(c)$ is the term involving the
product of the $i$-th and the $j$-th element, $\delta_i(c)$ is the term involving the left
action of the $i$-th element, and $\partial(c)$ is the term involving the right action of the
$(n+1)$-th element. As the bimodule is symmetric, the term involving the right action can
be counted among the terms involving the left actions.

We have to show that $\dl^\bullet(\,{\mathcal F}^p\CCE_\rel^{n}(\gf,M))\subseteq{\mathcal F}^p
\CCE_\rel^{n+1}(\gf,M)$. We thus consider the different terms of $\dl^\bullet(c)$ with
two equal elements as arguments in the first $p+1$ positions and have to show that all
terms are zero. For $\dl_{ij}(c)$ with $i,j\leq p+1$, the assertion is clear because either
the two equal elements do not occur in the  product, and then it is correct, or at least one
of them occurs, and then from the product terms of the sum of $\dl_{ij}$ and $\dl_{ij+1}$
(or $\dl_{ij}$ and $\dl_{ij-1}$) we obtain an element $x_ix\otimes x+x\otimes x_ix$, which
is a sum of symmetric elements thanks to $$x_ix\otimes x+x\otimes x_ix=(x_ix+x)\otimes
(x_ix+x)-x_ix\otimes x_ix-x\otimes x\,.$$

Even more elementary, the assertion holds for $\dl_{ij}(c)$ with $i,j\geq p+1$. For $\dl_{ij}
(c)$ with $i\leq p+1$ and $j\geq p+2$, the assertion is clear in case $x_i$ is not one of the
equal elements. In case it is, the two terms corresponding to the product action of the two
equal elements cancel as they are equal and have different sign.

For the action terms $\delta_i(c)$ the reasoning is similar. In case $i\leq p+1$, either the two
equal elements do not occur and the assertion holds, or both occur and cancel each other
because of the alternating sign. For $\delta_i(c)$ with $i\geq p+2$, the assertion is clear in
any case.
\end{proof}

The lemma implies that there is a spectral sequence of a filtered cochain complex associated
to this filtration which thanks to (\ref{finite_filtration_1}) converges  in the strong (i.e., finite)
sense to $\HCE_\rel^n(\gf,M)$. 

The next step is then to compute the $0$-th term of this spectral sequence, i.e., the associated
graded vector space of the filtration $$E_0^{p,q}:= {\mathcal F}^p\CCE_\rel^{p+q}(\gf,M)
\,/\,{\mathcal F}^{p+1}\CCE_\rel^{p+q}(\gf,M)\,.$$ Observe that
\begin{multline*}
{\mathcal F}^p\CCE_\rel^{p+q}(\gf,M)=\\
\{c\in\CL^{p+q+2}(\gf,M_s)\mid c(x_1,\ldots,x_{p+q+2})=0\mbox{ if }\exists\,j\leq p+1:\,
x_{j-1}=x_j\}\,/\,\CCE^{p+q+2}(\gf,M).
\end{multline*}
In the quotient space $E_0^{p,q}$, the term $\CCE^{p+q+2}(\gf,M)$, by which both filtration
spaces are  divided, disappears.
 
Observe that the filtration can be expressed as  $${\mathcal F}^p\CCE_\rel^{p+q}(\gf,M)
=\{[c]\in\CCE_\rel^{p+q}(\gf,M)\,|\,c|_{\Ker(\otimes^{p+1}\gf\to\Lambda^{p+1}\gf)
\otimes(\otimes^{q+1}\gf)}=0\}\,.$$ This is useful, because by elementary linear algebra,
we have $$F^{\perp}/\,G^{\perp}=\Hom_\F(G/F,M)\,,$$ where $F^{\perp}:=\{f:E\to M
\mid f_{\vert F}=0\}$ and $G^{\perp}:=\{f:E\to M\mid f_{\vert G}=0\}$ for $F\subseteq
G\subseteq E$.

In order to be able to find $E_0^{p,q}$, we therefore have to compute $$\Ker(\otimes^{p+2}\gf\to\Lambda^{p+2}\gf)\otimes(\otimes^{q}\gf)\,/\,\Ker
(\otimes^{p+1}\gf\to\Lambda^{p+1}\gf)\otimes(\otimes^{q+1}\gf)\,.$$ By using the
isomorphism (see the proof of Theorem A in \cite{P}) $$\Ker(\otimes^{p+2}\gf\to
\Lambda^{p+2}\gf)\,/\,\Ker(\otimes^{p+2}\gf\to\Lambda^{p+1}\gf\otimes\gf)
\cong\Ker(\Lambda^{p+1}\gf\otimes\gf\to\Lambda^{p+2}\gf)$$ and by applying
that $\Hom_\F$ and $\otimes$ are adjoint functors, we obtain that
\begin{eqnarray*}
E_0^{p,q} & = & \Hom_\F(\Ker(\Lambda^{p+1}\gf\otimes\gf\to\Lambda^{p+2}\gf)
\otimes\CL_{q}(\gf),M)\\
& = & \Hom_\F(\Ker(\Lambda^{p+1}\gf\otimes\gf\to\Lambda^{p+2}\gf),\Hom_\F(\CL_{q}(\gf),M))\\
& = & \Hom_\F(\CR_p(\gf),\CL^{q}(\gf,M))\,,
\end{eqnarray*}
by definition of the chain complex $\CR_\bullet(\gf)$ as the kernel of $m_\bullet$
(up to a degree shift). 
  
In the case of a finite-dimensional Lie algebra $\gf$, we can use the isomorphism
$\Hom_\F(U\otimes V,W)\cong U^*\otimes\Hom_\F(V,W)$ to write the $E_0$-term as $$E_0^{p,q}=[\Ker(\Lambda^{p+1}\gf\otimes\gf\to\Lambda^{p+2}
\gf)]^*\otimes\CL^{q}(\gf,M)\,.$$ In this particular case, one may observe that the
first tensor factor is the kernel of the exterior multiplication map $m_\bullet$, and thus
$$\Ker(\Lambda^{p+1}\gf\otimes\gf\to\Lambda^{p+2}\gf)^*=\Ker(m_\bullet)[1]^*=
\coker(m^\bullet)[-1]=\CR^p(\gf)\,.$$ Therefore the term $E_0^{p,q}$ takes in this
particular case the form $$E_0^{p,q}=\CR^p(\gf)\otimes\CL^{q}(\gf,M)\,.$$

Next, we will determine the differential on $E_0^{p,q}$:

\begin{lem}
The differential $\dl_0$ on $E_0^{p,q}\cong\Hom_\F(\CR_p(\gf),\CL^{q}(\gf,M))$
is induced by the coboundary operator $\dl^\bullet$ on $\CCE_\rel^\bullet(\gf,M)$.
More precisely, we have that $$d_0^{p,q}(f):=\dl^q_{\CL^{q}(\gf,M)}\circ f$$ for
every linear transformation $f\in\Hom_\F(\CR_p(\gf),\CL^{q}(\gf,M))$.
\end{lem}

\begin{proof}
By definition, the differential $\dl_0$ of the spectral sequence is the differential which is induced
by the Leibniz coboundary map $\dl^\bullet$ on the associated graded quotients $$\dl_0:
{\mathcal F}^p\CCE_\rel^{p+q}(\gf,M)\,/\,{\mathcal F}^{p+1}\CCE_\rel^{p+q}(\gf,M)\to
{\mathcal F}^p\CCE_\rel^{p+q+1}(\gf,M)\,/\,{\mathcal F}^{p+1}\CCE_\rel^{p+q+1}(\gf,M)\,.$$
In order to examine which terms $\dl_{ij}(c)$, $\delta_i(c)$ and $\partial(c)$ are zero for a cochain 
$c\in{\mathcal F}^p\CCE_\rel^{p+q}(\gf,M)$, we have to insert two consecutive equal elements in
the arguments of $c$ within the first $p+2$ arguments.

Now, by the same reasoning as in the proof of Lemma \ref{lemma_filtration_compatible}, the terms
$\dl_{ij}(c)$ vanish in case $i,j\leq p+2$, because in case the equal elements are not involved, the
formula for $\dl_{ij}(c)$ diminishes the number of arguments by one, and as $c$ is of degree $p$ in
the filtration, this then gives zero. In case the elements occur, they create once again a symmetric
element of the form $x_ix\otimes x+x\otimes  x_ix$. Also for $\dl_{ij}(c)$ with $i\leq p+2$ and
$j\geq p+3$, the terms are zero when the equal elements are not involved, and are zero in addition
with $\dl_{ij+1}(c)$ (or $\dl_{ij-1}(c)$), in case of multiplying with one of the equal elements. The
terms $\delta_i(c)$ for $i\leq p+1$ vanish as the corresponding formula diminishes the number
of arguments by one in case the equal elements do not occur, and annihilate each other in case
they occur.

Thus, there remain the terms $\dl_{ij}(c)$ with $i,j\geq p+3$, $\delta_i(c)$ with $i\geq p+3$,
and $\partial(c)$, which form together the coboundary map of the cochain complex $\CL^\bullet
(\gf,M)$.
\end{proof}

Consequently, in the general case we obtain for the first term of the spectral sequence:
$$E_1^{p,q}=\Hom_\F(\CR_p(\gf),\HL^{q}(\gf,M_s))\,,$$
and in the case of a finite-dimensional Lie algebra $\gf$ we have that
$$E_1^{p,q}=\CR^p(\gf)\otimes\HL^{q}(\gf,M_s))\,.$$

Now we proceed to identify the differential $\dl_1$ on $E_1^{p,q}$. The differential
$\dl_1$ is still induced by the Leibniz coboundary map on the filtered cochain complex.
As the classes in $\HL^{q}(\gf,M_s)$ are represented by cocycles, the part of the
Leibniz coboundary operator $\dl^q_{\CL^q(\gf,M_s)}$ constituting the differential
must be zero. The action of one of the remaining terms on $\HL^{q}(\gf,M_s)$ must
also be zero since the Cartan relations for Leibniz cohomology (due to Loday and
Pirashvili \cite[Proposition 3.1]{LP}) imply that a Leibniz algebra acts trivially on its
cohomology. (For the reader interested in left Leibniz algebras, a proof of these
formulas adapted to this case can be found in \cite[Proposition 1.3.2]{C}.) Note that
the Cartan relations only hold for $q\geq 1$. But as the Leibniz $\gf$-bimodule $M$
is symmetric, the action of $\gf$ on $\HL^{0}(\gf,M_s)$ is also trivial. Therefore, the
remaining terms of the differential constitute the coboundary operator $\dl^p_{\CL^p
(\gf)}$ on $\CR^p(\gf)$. 

Consequently, in the general case we obtain for the second term of the spectral sequence:
$$E_2^{p,q}=\HR^p(\gf,\HL^{q}(\gf,M_s))\,,$$ where the right-hand side denotes
the $\HR$-cohomology with values in the trivial $\gf$-module $\HL^\bullet(\gf,M_s)$.
It is the cohomology of the cochain complex arising from applying the exact
functor $\Hom_\F(-,\HL^\bullet(\gf,M_s))$ to the chain complex $\CR_\bullet(\gf)$.
It follows from the exactness of this functor and the Universal Coefficient Theorem
(for example, see \cite[Theorem 3.6.5]{HA}) that we can express
the $E_2$-term as $$E_2^{p,q}=\Hom_\F(\HR_p(\gf),\HL^{q}(\gf,M_s))\,,$$ and
in the special case of a finite-dimensional Lie algebra $\gf$, we have that $$E_2^{p,q}
=\HR^p(\gf)\otimes\HL^{q}(\gf,M_s)\,.$$

This discussion proves the following result which (up to dualization) is Theorem~A in
\cite{P}. In the case of trivial coefficients (and possibly topological Fr\'echet Lie
algebras) the second part of Theorem \ref{theorem_A} has also been obtained by
Lodder (see \cite[Theorem 2.10]{Lo}).

\begin{thm}\label{theorem_A}
Let $\gf$ be a Lie algebra, and let $M$ be a left $\gf$-module considered as a
symmetric Leibniz $\gf$-bimodule $M_s$. Then there is a spectral sequence converging
to $\HCE^\bullet_\rel(\gf,M)$ with second term $$E_2^{p,q}=\Hom_\F(\HR_p(\gf),
\HL^{q}(\gf,M_s))\,.$$ Moreover, if $\gf$ is finite dimensional, then the $E_2$-term
of this spectral sequence can be written as $$E_2^{p,q}=\HR^p(\gf)\otimes \HL^{q}
(\gf,M_s))\,.$$
\end{thm}

\noindent {\bf Remark.} As the spectral sequence of Theorem \ref{theorem_A} is
the spectral sequence of a filtered cochain complex, the higher differentials in this
spectral sequence are again induced by the Leibniz coboundary operator $\dl^{\bullet}$.
We will see in Example B below an instance of a concrete computation of the differential
$\dl_2$.
\vspace{.3cm}

Our main application of the spectral sequence will be the next theorem which is a refinement of
the cohomological analogue of \cite[Corollary 1.3]{P}:

\begin{thm}\label{vanlie}
Let $\gf$ be a Lie algebra, let $M$ be a left $\gf$-module considered as a symmetric Leibniz
$\gf$-bimodule $M_s$, and let $n$ be a non-negative integer. If $\HCE^k(\gf,M)=0$ for every
integer $k$ with $0\le k\le n$, then $\HL^k(\gf,M_s)=0$ for every integer $k$ with $0\le k\le
n$ and $\HL^{n+1}(\gf,M_s)\cong\HCE^{n+1}(\gf,M)$ as well as $\HL^{n+2}(\gf,M_s)\cong
\HCE^{n+2}(\gf,M)$. In particular, $\HCE^\bullet(\gf,M)=0$ implies that $\HL^\bullet(\gf,M_s)
=0$.
\end{thm}

\begin{proof}
The proof follows the proof of Corollary 1.3 in \cite{P} very closely.

According to Proposition \ref{lesrelcoh}, it suffices to prove that $\HCE^k(\gf,M)=0$ for
every integer $k$ with $0\le k\le n$ implies that $\HCE_\rel^n (\gf,M)=0$ for every integer
$k$ with $0\le k\le n$. We proceed by induction on $n$. In the case $n=0$, the hypothesis
yields that $E_2^{0,0}=0$ for the second term of the spectral sequence of Theorem~\ref{theorem_A},
and therefore we obtain from the convergence of the spectral sequence that $\HCE_\rel^0
(\gf,M)=0$ which initializes the induction.

So suppose now that $n\ge 1$ and $\HCE^k(\gf,M_s)=0$ for every integer $k$ with $0\le
k\le n+1$. By induction hypothesis, we obtain that $\HCE_\rel^n (\gf,M)=0$ for every
integer $k$ with $0\le k\le n$. Hence it follows from Proposition \ref{lesrelcoh} that
$\HL^k(\gf,M_s)=0$ for every integer $k$ with $0\le k\le n$ and $\HL^{n+1}(\gf,M_s)
\cong\HCE^{n+1}(\gf,M)=0$. Consequently, the second term $E_2^{p,q}$ of the
spectral sequence in Theorem \ref{theorem_A} is zero for $p+q\leq n+1$, and therefore
$\HCE_\rel^{n+1}(\gf,M)=0$.

Finally, the isomorphisms in degree $n+1$ and $n+2$, respectively, are an immediate
consequence of Proposition \ref{lesrelcoh}.
\end{proof}

\noindent {\bf Remark.} Note that the converse of Theorem \ref{vanlie} is also true, namely,
$\HCE^k(\gf,M)=0$ for every integer $k$ with $0\le k\le n$ if, and only if, $\HL^k(\gf,M_s)=0$
for every integer $k$ with $0\le k\le n$. In particular, $\HCE^\bullet(\gf,M)=0$ if, and only if,
$\HL^\bullet(\gf,M_s)=0$.
\vspace{.3cm}

Next, we illustrate the use of the spectral sequence of Theorem \ref{theorem_A} and the
associated long exact sequences (see Propositions \ref{lescoadj} and \ref{lesrelcoh}) by
two examples. We begin by computing the Leibniz cohomology of the smallest non-nilpotent
Lie algebra with coefficients in an arbitrary irreducible Leibniz bimodule (see also \cite[Example
1.4\,i)]{P} for trivial coeffcients in characteristic $\ne 2$). Note that for a ground field of
characteristic 2 the Leibniz cohomology of this Lie algebra is far more complicated than for
a field of characteristic $\ne 2$.

In the case that the irreducible Leibniz bimodule is of finite dimension $\ne 1$ we can prove more
generally for an arbitrary supersolvable Lie algebra the following vanishing result.

\begin{pro}\label{vansupsolv}
Let $\gf$ be a finite-dimensional supersolvable Lie algebra over a field $\F$, and let $M$
be a finite-dimensional irreducible Leibniz $\gf$-bimodule such that $\dim_\F M\ne 1$. Then
$\HL^n(\gf,M)=0$ for every positive integer $n$. Moreover, if $M$ is symmetric, then
$\HL^n(\gf,M)=0$ for every non-negative integer $n$.
\end{pro}

\begin{proof}
If $M$ is symmetric, then the assertion is an immediate consequence of Theorem \ref{vanlie}
in conjunction with \cite[Theorem 3]{B1}.

Now suppose that $M$ is not symmetric. Then it follows from \cite[Theorem 3.14]{F} that $M$
is anti-symmetric. We obtain from Lemma \ref{antisym}\,(b) that $$\HL^n(\gf,M)\cong\HL^{n-1}
(\gf,\Hom_\F(\gf,M)_s)\cong\HL^{n-1}(\gf,(\gf^*\otimes M)_s)$$ for every positive integer $n$.
By definition of supersolvability, the adjoint $\gf$-module has a composition series $$\gf_{\ad,
\ell}=\gf_n\supset\gf_{n-1}\supset\cdots\supset\gf_1\supset\gf_0=0$$ such that $\dim_\F\gf_j/
\gf_{j-1}=1$ for every integer $1\le j\le n$. From the short exact sequences $0\to\gf_{j-1}\to
\gf_j\to\gf_j/\gf_{j-1}\to 0$, we obtain by dualizing, tensoring each term with $M$, and symmetrizing
the short exact sequences: $$0\to[(\gf_j/\gf_{j-1})^*\otimes M]_s\to(\gf_j^*\otimes M)_s\to
(\gf_{j-1}^*\otimes M)_s\to 0$$ for every integer $1\le j\le n$. Since $M$ is irreducible and
$\dim_\F\gf_j/\gf_{j-1}=1$, we conclude that $[(\gf_j/\gf_{j-1})^*\otimes M]_s$ is an irreducible
symmetric Leibniz $\gf$-bimodule. Moreover, we have that $\dim_\F[(\gf_j/\gf_{j-1})^*\otimes M]_s\ne
1$ as $\dim_\F M\ne 1$. Hence we obtain inductively from the long exact cohomology sequence
that $\HL^n(\gf,M)\cong\HL^{n-1}(\gf,(\gf^*\otimes M)_s)=0$ for every positive integer $n$.
\end{proof}

\noindent {\bf Remark.} It follows from Lie's theorem that every finite-dimensional irreducible
Leibniz bimodule of a finite-dimensional solvable Lie algebra over an algebraically closed field
of characteristic zero is one-dimensional (see \cite[Corollary 4.1\,A]{H}). Consequently, in this
case the hypothesis of Proposition \ref{vansupsolv} is never satisfied, and thus this result is
only applicable over non-algebraically closed fields of characteristic zero or over fields of prime
characteristic.
\vspace{.3cm}

By virtue of Proposition \ref{vansupsolv}, in the next example it is enough to consider
one-dimensional Leibniz bimodules.
\vspace{.3cm}

\noindent {\bf Example A.} Let $\F$ denote an arbitrary field, and let $\af:=\F h\oplus\F e$
be the non-abelian two-dimensional Lie algebra over $\F$ with multiplication determined by
$he=e=-eh$. For any scalar $\lambda\in\F$ one can define a one-dimensional left $\af$-module
$F_\lambda:=\F 1_\lambda$ with $\af$-action defined by $h\cdot 1_\lambda:=\lambda
1_\lambda$ and $e\cdot 1_\lambda:=0$. Then the Chevalley-Eilenberg cohomology of $\af$
with coefficients in $F_\lambda$ is as follows:
\begin{eqnarray*}
\HCE^n(\af,F_\lambda)\cong
\left\{
\begin{array}{cl}
\F & \mbox{if }\lambda=0\mbox{ and }n=0,1\mbox{ or }\lambda=1\mbox{ and }n=1,2\\ 
0 & \hspace{2.5cm}\mbox{otherwise}\,.
\end{array}
\right.
\end{eqnarray*}
In particular, if $\lambda\ne 0,1$, then $\HCE^\bullet(\af,F_\lambda)=0$.

First, let us consider $F_\lambda$ as a symmetric Leibniz $\af$-bimodule $(F_\lambda)_s$. Then it
follows from Theorem \ref{vanlie} that $\HL^\bullet(\af,(F_\lambda)_s)=0$ for $\lambda\ne
0,1$.

In order to be able to compute the Leibniz cohomology for $\lambda=0,1$, and for the
anti-symmetric Leibniz $\af$-bimodules $(F_\lambda)_a$, let $M$ be an arbitrary left $\af$-module
considered as a symmetric Leibniz $\af$-bimodule $M_s$. Since $\HCE^n(\af,M)=0$ for every integer
$n\ge 3$, we obtain from Proposition \ref{lesrelcoh} the short exact sequence
\begin{equation}\label{leibnizses}
0\to\HCE^2(\af,M)\to\HL^2(\af,M_s)\to\HCE_\rel^0(\af,M)\to 0
\end{equation}
and the isomorphisms
\begin{equation}\label{leibnizreduc}
\HL^n(\af,M_s)\cong\HCE_\rel^{n-2}(\af,M)\hspace{.1cm}\mbox{for every integer }n\ge 3\,.
\end{equation}
Moreover, we have that $\HL^0(\af,M_s)\cong M^\af$ and $\HL^1(\af,M_s)\cong\HCE^1(\af,M)$.

For the computation of the relative cohomology spaces $\HCE_\rel^n(\af,M)$ we
need the coadjoint Chevalley-Eilenberg cohomology of $\af$. It is easy to verify that
$$\dim_\F\HCE^0(\af,\af^*)=1\,,$$
\begin{eqnarray*}
\dim_\F\HCE^1(\af,\af^*)=
\left\{
\begin{array}{cl}
2 & \mbox{if }\ch(\F)=2\\
1 & \mbox{if }\ch(\F)\ne 2\,,
\end{array}
\right.
\end{eqnarray*}
and
\begin{eqnarray*}
\dim_\F\HCE^2(\af,\af^*)=
\left\{
\begin{array}{cl}
1 & \mbox{if }\ch(\F)=2\\
0 & \mbox{if }\ch(\F)\ne 2\,.
\end{array}
\right.
\end{eqnarray*}
Consequently, we have to consider the cases $\ch(\F)=2$ and $\ch(\F)\ne 2$ differently.

Let us first assume that $\ch(\F)\ne 2$. Then it follows from Proposition \ref{lescoadj} that
$\HR^0(\af)\cong\HCE^1(\af,\af^*)\cong\F$ and $\HR^n(\af)=0$ for every integer
$n\ge 1$. Hence we derive from the spectral sequence of Theorem \ref{theorem_A}
that $\HCE_\rel^n(\af,M)\cong\HL^n(\af,M_s)$ for every non-negative integer $n$. In
conclusion, we obtain from (\ref{leibnizses}) and (\ref{leibnizreduc}) that
\begin{equation}\label{leibniz2}
\HL^2(\af,M_s)\cong M^\af\oplus\HCE^2(\af,M)
\end{equation}
and
\begin{equation}\label{leibnizhigh}
\HL^n(\af,M_s)\cong\HL^{n-2}(\af,M_s)\hspace{.1cm}\mbox{for every integer }n\ge 3\,.
\end{equation}

As an immediate consequence, in the case of $\ch(\F)\ne 2$ we deduce that $\dim_\F\HL^n
(\af,\F)=1$ for every non-negative integer $n$ and
\begin{eqnarray*}
\dim_\F\HL^n(\af,(F_1)_s)=
\left\{
\begin{array}{cl}
0 & \mbox{if }n=0\\ 
1 & \mbox{if }n>0\,.
\end{array}
\right.
\end{eqnarray*}

In summary, we have for the Leibniz cohomology of $\af$ over a field $\F$ of characteristic
$\ne 2$ with coefficients in a one-dimensional symmetric Leibniz bimodule that
\begin{eqnarray*}
\dim_\F\HL^n(\af,(F_\lambda)_s)=
\left\{
\begin{array}{cl}
1 & \mbox{if }\lambda=0\mbox{ and }n\mbox{ is arbitrary or if }\lambda=1\mbox{ and }n>0\\
0 & \hspace{3cm}\mbox{otherwise}\,.
\end{array}
\right.
\end{eqnarray*}

Next, let us assume that $\ch(\F)=2$. Then it follows from Proposition \ref{lescoadj} that
$\HR^0(\af)\cong\HCE^1(\af,\af^*)\cong\F^2$, $\HR^1(\af)\cong\HCE^2(\af,\af^*)
\cong\F$, and $\HR^n(\af)=0$ for every integer $n\ge 2$. Hence in the spectral sequence
of Theorem \ref{theorem_A}, we have only two non-zero columns, namely the $p=0$ and
the $p=1$ column. In the $p=0$ column, we have spaces $\F^2\otimes\HL^q(\af,M_s)
\cong\HL^q(\af,M_s)\oplus\HL^q(\af,M_s)$, while in the $p=1$ column, we have just
$\HL^q(\af,M_s)$ for every integer $q\geq 0$. Therefore, the spectral sequence degenerates
at the term $E_2$, and for every integer $n\geq 1$ we obtain that
$$
\HCE^n_\rel(\af,M)\cong\HL^n(\af,M_s)\oplus\HL^n(\af,M_s)\oplus\HL^{n-1}(\af,M_s)\,,
$$
and
$$
\HCE^0_\rel(\af,M)\cong E_2^{0,0}\cong\HL^0(\af,M_s)\oplus\HL^0(\af,M_s)\cong M^\af
\oplus M^\af\,.
$$
This, together with (\ref{leibnizses}), (\ref{leibnizreduc}), and induction yields the recursive
relation
\begin{equation}\label{fibonacci}
\HL^n(\af,M_s)\cong\HL^{n-1}(\af,M_s)\oplus\HL^{n-2}(\af,M_s)\hspace{.1cm}\mbox{for 
every integer }n\ge 2\,.
\end{equation}

As a consequence, we obtain for the Leibniz cohomology of $\af$ over a field $\F$ of characteristic
$2$ with coefficients in a one-dimensional symmetric Leibniz bimodule that
\begin{eqnarray*}
\dim_\F\HL^n(\af,(F_\lambda)_s)=
\left\{
\begin{array}{cl}
f_{n+1} & \mbox{if }\lambda=0\\
f_n & \mbox{if }\lambda=1\\
0 & \mbox{otherwise}
\end{array}
\right.
\end{eqnarray*}
for every non-negative integer $n$, where $f_n$ denotes the $n$-th term of the
standard Fibonacci sequence given by $f_0:=0$, $f_1:=1$, and $f_n:=f_{n-1}+f_{n-2}$
for every integer $n\geq 2$. In particular, we have that $$\HL^n(\af,(F_1)_s)\cong
\HL^{n-1}(\af,\F)\hspace{.1cm}\mbox{for every integer }n\ge 1\,.$$

Next, let us consider $F_\lambda$ as an anti-symmetric Leibniz $\af$-bimodule $(F_\lambda)_a$
with the same left $\af$-action as above and with the trivial right $\af$-action (see Section
\ref{prelim}). Then we conclude from Lemma \ref{antisym}\,(b) that $$\dim_\F\HL^0(\af,
(F_\lambda)_a)=1\mbox{ for every }\lambda\in\F\,.$$

Let us now compute $\HL^n(\af,(F_\lambda)_a)$ for any integer $n\ge 1$. It follows from
Lemma~\ref{antisym}\,(b) that
\begin{equation}\label{leibnizalt}
\HL^n(\af,(F_\lambda)_a)\cong\HL^{n-1}(\af,\Hom_\F(\af,F_\lambda)_s)\cong\HL^{n-1}
(\af,(\af^*\otimes F_\lambda)_s)\,.
\end{equation}
A straightforward computation shows that $$0\to F_\lambda\to\af^*\otimes F_\lambda\to
F_{\lambda-1}\to 0$$ is a short exact sequence of left $\af$-modules. Then we obtain from
the long exact cohomology sequence and another straightforward computation in the case
$\lambda=1$:
\begin{eqnarray*}
\dim_\F(\af^*\otimes F_\lambda)^\af=
\left\{
\begin{array}{cl}
1 & \mbox{if }\lambda=0\\ 
0 & \mbox{otherwise}\,,
\end{array}
\right.
\end{eqnarray*}
\begin{eqnarray*}
\dim_\F\HCE^1(\af,\af^*\otimes F_\lambda)=
\left\{
\begin{array}{cl}
2 & \mbox{if }\lambda=0\mbox{ and }\ch(\F)=2\\
1 & \mbox{if }\lambda=0,2\mbox{ and }\ch(\F)\ne 2\,,\\ 
0 & \hspace{1cm}\mbox{otherwise}
\end{array}
\right.
\end{eqnarray*}
and
\begin{eqnarray*}
\dim_\F\HCE^2(\af,\af^*\otimes F_\lambda)=
\left\{
\begin{array}{cl}
1 & \mbox{if }\lambda=0\mbox{ and }\ch(\F)=2\mbox{ or }\lambda=2
\mbox{ and }\ch(\F)\ne 2\\ 
0 & \hspace{3cm}\mbox{otherwise}\,.
\end{array}
\right.
\end{eqnarray*}

If $\ch(\F)\ne 2$, we conclude by applying (\ref{leibnizalt}) and (\ref{leibniz2}) to the
symmetric Leibniz $\af$-bimodule $M_s:=(\af^*\otimes F_\lambda)_s$ that
\begin{eqnarray*}
\dim_\F\HL^1(\af,(F_\lambda)_a)=
\left\{
\begin{array}{cl}
1 & \mbox{if }\lambda=0\\ 
0 & \mbox{otherwise}\,,
\end{array}
\right.
\end{eqnarray*}
and
\begin{eqnarray*}
\dim_\F\HL^3(\af,(F_\lambda)_a)=\dim_\F\HL^2(\af,(F_\lambda)_a)=
\left\{
\begin{array}{cl}
1 & \mbox{if }\lambda=0,2\\ 
0 & \mbox{otherwise}\,,
\end{array}
\right.
\end{eqnarray*}
Finally, we use (\ref{leibnizhigh}) to deduce for every integer $n\ge 2$:
\begin{eqnarray*}
\dim_\F\HL^n(\af,(F_\lambda)_a)=
\left\{
\begin{array}{cl}
1 & \mbox{if }\lambda=0,2\\ 
0 & \mbox{otherwise}\,.
\end{array}
\right.
\end{eqnarray*}

In summary, we have for the Leibniz cohomology of $\af$ over a field $\F$ of characteristic
$\ne 2$ with coefficients in a one-dimensional anti-symmetric Leibniz bimodule that
\begin{eqnarray*}
\dim_\F\HL^n(\af,(F_\lambda)_a)=
\left\{
\begin{array}{cl}
1 & \mbox{if }\lambda=0\mbox{ and }n\mbox{ is arbitrary or }\lambda=2\mbox{ and }n\ge 2\\
   & \mbox{or }n=0\mbox{ and }\lambda\mbox{ is arbitrary}\\
0 & \hspace{3cm}\mbox{otherwise}\,.
\end{array}
\right.
\end{eqnarray*}

If $\ch(\F)=2$, we obtain by applying (\ref{leibnizalt}) and (\ref{fibonacci}):
\begin{eqnarray*}
\dim_\F\HL^n(\af,(F_\lambda)_a)=
\left\{
\begin{array}{cl}
1 & \mbox{if }n=0\mbox{ and }\lambda\mbox{ is arbitrary}\\
f_{n+1} & \mbox{if }\lambda=0\mbox{ and }n\mbox{ is arbitrary}\\
0 & \hspace{1cm}\mbox{otherwise}\,.
\end{array}
\right.
\end{eqnarray*}
%This again shows that Leibniz cohomology behaves more uniformly than Chevalley-Eilenberg
%cohomology.
\vspace{.2cm}

\noindent {\bf Remark.} Since every invariant symmetric bilinear form on $\af$ is a multiple
of the Killing form, we have that $[S^2(\af)^*]^\af\cong\F$. On the other hand, from the
computations in Example A we obtain when $\ch(\F)=2$ that $$\HCE_\rel^0(\af,\F)\cong
\HR^0(\af)\cong\HCE^1(\af,\af^*)\cong\F^2\,.$$ This shows that, in general, $\HR^0
(\af)\not\cong[S^2(\af)^*]^\af$ and $\HCE_\rel^0(\af,\F)\not\cong[S^2(\af)^*]^\af$.

Moreover, the computations in Example A show that $\dim_\F\HL^3(\af,\F)=3$, but $\dim_\F
\HCE^2(\af,\af^*)\le 1$. Hence the fifth terms of the five-term exact sequences in Proposition
\ref{lescoadj} and in Proposition \ref{lesrelcoh} for $M:=\F$ are not necessarily isomorphic.
\vspace{.3cm}

Similar to Proposition \ref{vansupsolv}, we have the following general vanishing result:

\begin{pro}\label{vannilp}
Let $\gf$ be a finite-dimensional nilpotent Lie algebra, and let $M$ be a finite-dimensional
non-trivial irreducible Leibniz $\gf$-bimodule. Then $\HL^n(\gf,M)=0$ for every positive
integer $n$. Moreover, if $M$ is symmetric, then $\HL^n(\gf,M)=0$ for every non-negative
integer $n$.
\end{pro}

\begin{proof}
If $M$ is symmetric, then the assertion is an immediate consequence of Theorem \ref{vanlie}
and \cite[Lemma 3]{B1}.

According to \cite[Theorem 3.14]{F}, we can suppose that $M$ is anti-symmetric. We obtain
from Lemma \ref{antisym}\,(b) that $$\HL^n(\gf,M)\cong\HL^{n-1}(\gf,\Hom_\F(\gf,M)_s)
\cong\HL^{n-1}(\gf,(\gf^*\otimes M)_s)$$ for every positive integer $n$. By refining the
descending central series of $\gf$, one can construct a composition series $$\gf_{\ad,\ell}=
\gf_n\supset\gf_{n-1}\supset\cdots\supset\gf_1\supset\gf_0=0$$ of the adjoint $\gf$-module
such that $\gf_j/\gf_{j-1}$ is the trivial one-dimensional $\gf$-module $\F$ for every integer
$1\le j\le n$. From the short exact sequences $$0\to\gf_{j-1}\to\gf_j\to\F\to 0\,,$$ we obtain
by dualizing, tensoring each term with $M$, and symmetrizing the short exact sequences:
$$0\to M_s\to(\gf_j^*\otimes M)_s\to(\gf_{j-1}^*\otimes M)_s\to 0$$ for every integer
$1\le j\le n$. Since $M$ is a non-trivial irreducible left $\gf$-module, we conclude that $M_s$
is a non-trivial irreducible symmetric Leibniz $\gf$-bimodule. Hence we obtain inductively from
the long exact cohomology sequence that $\HL^n(\gf,M)\cong\HL^{n-1}(\gf,(\gf^*\otimes M)_s)
=0$ for every positive integer $n$.
\end{proof}

Since the Leibniz cohomology of an abelian Lie algebra with trivial coefficients is known,
in Example B we compute this cohomology for the smallest non-abelian nilpotent Lie algebra.
Note that in \cite[Example 1.4.\,iv)]{P} the corresponding Leibniz homology has been
computed. In fact, homology and cohomology of a finite-dimensional Leibniz algebra $\lf$
with {\it trivial} coefficients are isomorphic, as we have the duality isomorphism $\CL_{\bullet}
(\lf,\F)^*\cong\CL^{\bullet}(\lf,\F)$ already on the level of cochain complexes. Therefore
our results coincide with those of Pirashvili. We furthermore compute in Example C the
Leibniz cohomology of the smallest nilpotent non-Lie Leibniz algebra with coefficients in
the trivial Leibniz bimodule.
\vspace{.3cm}

\noindent {\bf Example B.} Let $\F$ denote an arbitrary field of characteristic $\ne 2$,
and let $\hf:=\F x\oplus\F y\oplus\F z$ be the three-dimensional Heisenberg algebra
over $\F$ with multiplication determined by $xy=z=-yx$. Then the Chevalley-Eilenberg
cohomology of $\hf$ with coefficients in the trivial module $\F$ is well-known:
\begin{eqnarray*}
\dim_\F\HCE^n(\hf,\F)=
\left\{
\begin{array}{cl}
1 & \mbox{if }n=0,3\\ 
2 & \mbox{if }n=1,2\,.\\
0 & \mbox{if }n\ge 4
\end{array}
\right.
\end{eqnarray*}
Consequently, we have that $\dim_\F\HL^0(\hf,\F)=1$ and $\dim_\F\HL^1(\hf,\F)=2$.

As $\HCE^n(\hf,\F)=0$ for every integer $n\ge 4$, we obtain from Proposition
\ref{lesrelcoh} the following six-term exact sequence:
$$0\to\HCE^2(\hf,\F)\to\HL^2(\hf,\F)\to\HCE_\rel^0(\hf,\F)\to\HCE^3(\hf,\F)\to
\HL^3(\hf,\F)\to\HCE_\rel^1(\hf,\F)\to 0$$ and $$\HL^n(\hf,\F)\cong\HCE_\rel^{n-2}
(\hf,\F)\hspace{.1cm}\mbox{for every integer }n\ge 4\,.$$

Since we assume that $\ch(\F)\ne 2$, it follows from the remark after Proposition~\ref{lesrelcoh}
that we can identify $\HCE_\rel^0(\hf,\F)$ with the space of invariant symmetric bilinear
forms on $\hf$ and the map $\HCE_\rel^0(\hf,\F)\to\HCE^3(\hf,\F)$ with the classical
Cartan-Koszul map. It is easy to see that the latter map is zero for the Heisenberg algebra,
which yields the surjectivity of the map $\HL^2(\hf,\F)\to\HCE_\rel^0(\hf,\F)$ and the
injectivity of the map $\HCE^3(\hf,\F)\to\HL^3(\hf,\F)$. As a consequence, we obtain
the following two short exact sequences: $$0\to\HCE^2(\hf,\F)\to\HL^2(\hf,\F)\to
\HCE_\rel^0(\hf,\F)\to 0\,,$$ $$0\to\HCE^3(\hf,\F)\to\HL^3(\hf,\F)\to\HCE_\rel^1
(\hf,\F)\to 0\,.$$

In order to compute $\HCE_\rel^0(\hf,\F)$ and $\HCE_\rel^1(\hf,\F)$, we need the
coadjoint Chevalley-Eilenberg cohomology of $\hf$. We have that $\dim_\F\HCE^0
(\hf,\hf^*)=2$, $\dim_\F\HCE^1(\hf,\hf^*)=5$, $\dim_\F\HCE^2(\hf,\hf^*)=4$,
and $\HCE^3(\hf,\hf^*)=1$. These dimensions can be computed directly, but for
the complex numbers as a ground field it also follows from the main result of \cite{M}
in conjunction with Poincar\'e duality (for the latter see \cite[Theorem 3.4]{W}).

Similar to the discussion of the consequences of Proposition \ref{lesrelcoh} above, we
obtain from Proposition \ref{lescoadj} the two short exact sequences $$0\to\HCE^2
(\hf,\F)\to\HCE^1(\hf,\hf^*)\to\HR^0(\hf)\to 0\,,$$ $$0\to\HCE^3(\hf,\F)\to\HCE^2
(\hf,\hf^*)\to\HR^1(\hf)\to 0\,,$$ the isomorphism $\HR^2(\hf)\cong\HCE^3(\hf,
\hf^*)$, and $\HR^n(\hf)=0$ for every integer $n\ge 3$.

From these two short exact sequences and the isomorphism we derive that $$\dim_\F
\HR^0(\hf)=\dim_\F\HCE^1(\hf,\hf^*)-\dim_\F\HCE^2(\hf,\F)=5-2=3\,,$$ $$\dim_\F
\HR^1(\hf)=\dim_\F\HCE^2(\hf,\hf^*)-\dim_\F\HCE^3(\hf,\F)=4-1=3\,,$$ and
$$\dim_\F\HR^2(\hf)=\dim_\F\HCE^3(\hf,\hf^*)=1\,,$$ respectively. Therefore we
obtain from $\HCE_\rel^0(\hf,\F)\cong\HR^0(\hf)$ that $$\dim_\F\HL^2(\hf,\F)=
\dim_\F\HCE^2(\hf,\F)+\dim_\F\HCE_\rel^0(\hf,\F)=2+3=5\,.$$

Now we want to apply the spectral sequence of Theorem \ref{theorem_A}. For this let us
compute the differential $$\dl_2^{0,1}:E_2^{0,1}=\HR^0(\hf)\otimes\HL^1(\hf,\F)\to
E_2^{2,0}=\HR^2(\hf)\otimes\HL^0(\hf,\F)\,.$$ In characteristic $\not=2$, an element
of $\HR^0(\hf)$ is an invariant symmetric bilinear form $\omega$. It is considered as a
$1$-cochain with values in $\hf^*$ and, as it is a representative of an element of a quotient
cochain complex, it is zero in case it is skew-symmetric in all entries. Take furthermore a
cocycle $c\in \CL^1(\hf,\F)$ and compute for three elements $r,s,t\in\hf$:
\begin{eqnarray*}
\dl^1(\omega\otimes c)(r,s,t)&=&\omega(rs,-)c(t)+\omega(s,-)c(rt)-\omega(r,-)c(st)+\\
&+&\omega(s,r-)c(t)-\omega(r,s-)c(t)+\omega(r,t-)c(s)
\end{eqnarray*}
Now as $c$ is a cocycle with trivial coefficients, $c$ vanishes on products, thus the second
and third terms are zero. Furthermore, the first and fourth term cancel by the invariance
of the form and skew-symmetry of the Lie product. We are left with the two last terms
$-\omega(r,s-)c(t)+\omega(r,t-)c(s)$, which are skew-symmetric in the three entries
of the element in $\HR^2(\hf)$ and vanish therefore as well. In conclusion, the differential
$\dl_2^{0,1}$ is zero, and we have that $$\HCE^1_{\rm rel}(\hf)=\HR^0(\hf)\otimes\HL^1
(\hf,\F)\oplus\HR^1(\hf)\otimes\HL^0(\hf,\F)\,.$$ This implies in turn
$$\dim_\F\HL^3(\hf,\F)=\dim_\F\HCE^3(\hf,\F)+\dim_\F\HCE_\rel^1(\hf,\F)=1+9=10\,.$$

It seems that all differentials $\dl_2$ are zero and thus that this scheme persists to yield 
the dimensions of the higher $\HCE_\rel^n(\hf,\F)$ and thus of $\HL^n(\hf,\F)$ (see the
dimension formula in \cite[Example 1.4\,iv)]{P}). 
\vspace{.2cm}

\noindent {\bf Remark.} As a by-product of the above computations we obtain
that the space $[S^2(\hf)^*]^\hf$ of invariant symmetric bilinear forms on $\hf$
is three-dimensional when $\ch(\F)\ne 2$.
\vspace{.3cm}

We proceed by proving an extension of a result by Fialowski, Magnin, and Mandal (see
Corollary 2 in \cite{FMM}), namely, the fact that the vanishing of the center $C(\gf)$ of
a Lie algebra $\gf$ implies $\HL^2(\gf,\gf_\ad)=\HCE^2(\gf,\gf)$, where $\gf_\ad$
denotes the adjoint Leibniz $\gf$-bimodule induced by the left and right multiplication
operator. Note that for Lie algebras, this bimodule is indeed symmetric.
%The spectral sequence of Theorem \ref{theorem_A} with adjoint coefficients has as its
%$E_2$-term $$E_2^{p,q}=\HR^p(\gf)\otimes\HL^q(\gf,\gf_\ad)\,.$$

Moreover, we observe that $\HCE^0(\gf,\gf)=C(\gf)$. Therefore, it is an immediate
consequence of the case $n=0$ of Theorem \ref{vanlie} that the vanishing of the center
implies $\HL^2(\gf,\gf_\ad)=\HCE^2(\gf,\gf)$. By the same token for $n=1$, we can
extend this result to complete Lie algebras, i.e., to those Lie algebras $\gf$ for which
$\HCE^0(\gf,\gf)=\HCE^1(\gf,\gf)=0$:

\begin{cor}\label{complet}
Let $\gf$ be a complete Lie algebra. Then $$\HL^2(\gf,\gf_\ad)\cong\HCE^2(\gf,\gf)
\,\,\,{\rm and}\,\,\,\HL^3(\gf,\gf_\ad)\cong\HCE^3(\gf,\gf)\,.$$
\end{cor}

A class of examples of complete Lie algebras over an algebraically closed field $\F$ of
characteristic zero consists of those finite-dimensional Lie algebras $\gf$ for which $\gf$
has the same dimension as its Lie algebra of derivations and $\dim_\F\gf/\gf^2>1$ (see
\cite[Proposition 3.1]{Car}). Another example is the two-sided Witt algebra over a field
of characteristic zero. Indeed, this infinite-dimensional simple Lie algebra is complete (see
\cite[Theorem~A.1.1]{ES}). Hence we obtain from \cite[Theorem 3.1]{S} and \cite[Theorem
4.1]{ES} in conjunction with the case $n=3$ of Theorem \ref{vanlie} the following result:

\begin{cor}\label{witt}
Let $\W:=\der(\F[t,t^{-1}])$ be the two-sided Witt algebra over a field $\F$ of characteristic
zero. Then $\HL^2(\W,\W_\ad)=0$ and $\HL^3(\W,\W_\ad)=0$. Moreover, $$\HL^4(\W,
\W_\ad)\cong\HCE^4(\W,\W)\,\,\,{\rm and}\,\,\,\HL^5(\W,\W_\ad)\cong\HCE^5(\W,\W)\,.$$
\end{cor}

\noindent {\bf Remark.} Very recently, Camacho, Omirov, and Kurbanbaev also proved that
the second adjoint Leibniz cohomology of $\W$ vanishes (see \cite[Theorem 4]{COK}) by
explicitly showing that every adjoint Leibniz 2-cocycle (resp.\ Leibniz 2-coboundary) is an
adjoint Chevalley-Eilenberg 2-cocycle (resp.\ Chevalley-Eilenberg 2-coboundary) for $\W$.
\vspace{.3cm}

We conclude this section with another application of Theorem \ref{vanlie}. Let $\F$ be an
algebraically closed field of characteristic zero, let $n$ be a non-negative integer, and let
$L_n(\F)\subseteq\F^{n^3}$ denote the affine variety of structure constants of the
$n$-dimensional left Leibniz algebras over $\F$ with respect to a fixed basis of $\F^n$. Then
the general linear group $\GL_n(\F)$ acts on $L_n(\F)$, and a point (= Leibniz multiplication
law) $\phi\in L_n(\F)$ is called {\em rigid\/} if the orbit $\GL_n(\F)\cdot\phi$ is open in $L_n
(\F)$. It follows from Corollary \ref{witt} in conjunction with \cite[Th\'eor\`eme 3]{Ba} that
the infinite-dimensional two-sided Witt algebra over an algebraically closed field of characteristic
zero is rigid as a Leibniz algebra.

It is well known that the Chevalley-Eilenberg cohomology of the non-abelian two-dimensional
Lie algebra with coefficients in the adjoint module vanishes. According to Theorem \ref{vanlie},
this implies that the corresponding Leibniz cohomology vanishes as well. Note that the non-abelian
two-dimensional Lie algebra is the standard Borel subalgebra of the split three-dimensional simple
Lie algebra $\slf_2$.

Similarly, by applying Theorem \ref{vanlie} in conjunction with \cite[Theorem 1]{T} (see also
\cite[Section 1]{LL2}) one obtains the following more general result in characteristic zero
(cf.\ also \cite[Proposition 2.3]{P} for the rigidity of parabolic subalgebras). Recall that a
subalgebra of a semi-simple Lie algebra $\gf$ is called {\em parabolic\/} if it contains a maximal
solvable (= Borel) subalgebra of $\gf$.

\begin{pro}\label{borel}
Let $\bbf$ be a parabolic subalgebra of a finite-dimensional semi-simple Lie algebra over a field
of characteristic zero. Then $\HL^n(\bbf,\bbf_\ad)=0$ for every non-negative integer $n$. In
particular, parabolic subalgebras of a finite-dimensional semi-simple Lie algebra over an algebraically
closed field of characteristic zero are rigid as Leibniz algebras.
\end{pro}

\noindent {\bf Remark.} It would be interesting to know whether Proposition \ref{borel} remains
valid in prime characteristic.
\vspace{.2cm}

%%%%%%%%%%%%%%%%%%%%%%%%%%%%%%%%%%%%%%%%%%%%%%%%%%

\section{A Hochschild-Serre type spectral sequence for Leibniz cohomology}\label{specseq}

%%%%%%%%%%%%%%%%%%%%%%%%%%%%%%%%%%%%%%%%%%%%%%%%%%

In this section we consider a Leibniz analogue of the Hochschild-Serre spectral sequence
for the Chevalley-Eilenberg cohomology of Lie algebras that converges to some relative
cohomology. It will play a predominant role in Section~\ref{semsim}. The homology
version of this spectral sequence with values in symmetric bimodules is due to Pirashvili
(see \cite[Theorem C]{P}). Our arguments follow Pirashvili very closely, but we include
all the details as it turns out that the spectral sequence holds more generally for arbitrary
bimodules.

Let $\pi:\lf\to\qf$ be an epimorphism of left Leibniz algebras, and let $M$ be a $\qf$-bimodule.
Then $M$ is also an $\lf$-bimodule via $\pi$. Moreover, the epimorphisms $\pi^{\otimes n}:
\lf^{\otimes n}\to\qf^{\otimes n}$ induce a monomorphism $\CL^\bullet(\qf,M)\to\CL^\bullet
(\lf,M)$ of cochain complexes. Now set $$\CL^\bullet(\lf\vert\qf,M):=\coker(\CL^\bullet
(\qf,M)\to\CL^\bullet(\lf,M))[-1]$$ and $$\HL^\bullet(\lf\vert\qf,M):=\HCE^\bullet(\CL^\bullet
(\lf\vert\qf,M))\,.$$ Then by applying the long exact cohomology sequence to the short exact
sequence $$0\to\CL^\bullet(\qf,M)\to\CL^\bullet(\lf,M)\to\CL^\bullet(\lf\vert\qf,M)[1]\to 0$$
of cochain complexes one obtains the following result (see also \cite[Proposition 4.1]{P} for the
corresponding result for Leibniz homology).

\begin{pro}\label{les}
For every epimorphism $\pi:\lf\to\qf$ of left Leibniz algebras and every $\qf$-bimodule $M$
there exists a long exact sequence
\begin{eqnarray*}
0 & \to & \HL^1(\qf,M)\to\HL^1(\lf,M)\to\HL^0(\lf\vert\qf,M)\\
& \to & \HL^2(\qf,M)\to\HL^2(\lf,M)\to\HL^1(\lf\vert\qf,M)\to\cdots
\end{eqnarray*}
\end{pro}

Let us now derive Pirashvili's analogue of the Hochschild-Serre spectral sequence for Leibniz
cohomology (see \cite[Theorem C]{P} for the homology version). While Pirashvili considers
only symmetric bimodules, we extend the dual of his spectral sequence to arbitrary bimodules.
The construction of this spectral sequence is very similar to the construction of the spectral
sequence in Theorem \ref{theorem_A}.
  
We consider the following filtration on the cochain complex $\CL^\bullet(\lf,M)[-1]$. 
$${\mathcal F}^p\CL^n(\lf,M)[-1]:=\{c\in\CL^{n+1}(\lf,M)\mid c(x_1,\ldots,x_{n+1})=0\mbox{ if }
\exists\,i\leq p:\,x_i\in \If\}\,.$$
This defines a finite decreasing filtration
\begin{multline} \label{finite_filtration_2}
{\mathcal F}^0\CL^n(\lf,M)[-1]=\CL^n(\lf,M)[-1]\supset{\mathcal F}^1\CL^n(\lf,M)[-1]\supset\cdots\\
\cdots\supset{\mathcal F}^{n+1}\CL^n(\lf,M)[-1]=\CL^n(\qf,M)[-1]\,,
\end{multline}
Then we have the following result:
  
\begin{lem}
This filtration is compatible with the Leibniz coboundary map $\dl^\bullet$. 
\end{lem}

\begin{proof}
We have to prove that $\dl^\bullet({\mathcal F}^p\CL^n(\lf,M)[-1])\subseteq{\mathcal F}^p
\CL^{n+1}(\lf,M)[-1]$. For this, we consider the different terms $\dl_{ij}(c)$, $\delta_i(c)$,
and $\partial(c)$, which constitute $\dl_0(c)$, where we have inserted an element of
$\If$ within the first $p$ arguments. The vanishing is clear for the terms $\dl_{ij}(c)$ with $i,j
\leq p$, because even if the element of $\If$ occurs in the product, the product will again be in
the ideal $\If$. The vanishing is also clear for the terms $\dl_{ij}(c)$ with $i,j\geq p+1$.
Concerning the terms $\dl_{ij}(c)$ with $i\leq p$ and $j\geq p+1$, we use the condition $\If
\subseteq C_\ell(\lf)$ to conclude that these are zero.

The action terms follow a similar pattern. The terms $\delta_i(c)$ with $i\leq p$ vanish, because
either the element of $\If$ occurs in the arguments, or it acts on $M$, which is zero by assumption.
The terms $\delta_i(c)$ with $i\geq p+1$ are zero for elementary reasons, as is the term $\partial(c)$.
\end{proof}

Thanks to (\ref{finite_filtration_2}), the associated spectral sequence converges 
in the strong (i.e., finite) sense to the cohomology $\HL^n(\lf|\qf,M)$ of the quotient cochain
complex $\CL^n(\lf,M)[-1]\,/CL^n(\qf,M)[-1]$, and we get for the 0-th term of the spectral
sequence
$$E_0^{p,q}=\Hom_\F(\qf^p\otimes\lf^{q+1},M)\,/\,\Hom_\F(\qf^{p+1}\otimes\lf^q,M)\,
\cong\,\Hom_\F(\qf^p\otimes\If\otimes\lf^q,M)\,,$$ where the isomorphism is induced by the
inclusion $\If\hookrightarrow\lf$. Since $\Hom_\F$ and $\otimes$ are adjoint functors, we obtain that
$$E_0^{p,q}=\Hom_\F(\qf^p\otimes\If,\Hom_\F(\lf^q,M))=\Hom_\F(\qf^p\otimes\If,\CL^q(\lf,M))\,.$$

Next, we will determine the differential on $E_0^{p,q}$:

\begin{lem}\label{lemma_3_5}
The differential $\dl_0$ on $E_0^{p,q}\cong\Hom_\F(\qf^p\otimes\If,\CL^q(\lf,M))$ is
induced by the coboundary operator $\dl^\bullet$ on $\CL^q(\lf,M)$. More precisely, we
have that $$d_0^{p,q}(f):=\dl^q_{\CL^{q}(\lf,M)}\circ f$$ for every linear transformation
$f\in\Hom_\F(\qf^p\otimes\If,\CL^q(\lf,M)))$.
\end{lem}

\begin{proof}
The differential
\begin{multline*}
\dl_0:{\mathcal F}^p\CL^n(\lf,M)[-1]\,/\,{\mathcal F}^{p+1}\CL^n(\lf,M)[-1]\to\\
\to{\mathcal F}^p\CL^{n+1}(\lf,M)[-1]\,/\,{\mathcal F}^{p+1}\CL^{n+1}(\lf,M)[-1]
\end{multline*}
is induced by $\dl^\bullet$. Thus, we have to examine which terms $\dl_{ij}(c)$, $\delta_i
(c)$, and $\partial(c)$ composing the value $\dl_0(c)$ are non-zero when we
put an element of  $\If$ within the first $p+1$ entries.

It is clear that $\dl_{ij}(c)=0$ for $i,j\leq p+1$ since this is true when the element of $\If$
is not involved in the product, as the number of elements is diminished by one, and it is also
true when the element of $\If$ is in the product because $\If$ is an ideal. We then have
$\dl_{ij}(c)=0$ for $i\leq p+1$ and $j\geq p+2$, since in case the element of $\If$ acts in
the product, it acts trivially because of $\If\subseteq C_\ell(\lf)$. Furthermore, the terms
$\delta_i(c)$ for $i\leq p+1$ are zero since elements of $\If$ act trivially on $M$.

Therefore, we are left with the terms composing the differential $\dl_0|_{\CL^q(\lf,M)}$.
\end{proof}

Consequently, the first term of the spectral sequence is
$$E_1^{p,q}=\Hom_\F(\qf^p\otimes\If,\HL^q(\lf,M))\,.$$

Now we proceed to determine the differential $\dl_1$ on $E_1^{p,q}$. It is again induced
by the Leibniz coboundary operator. As before, classes in $\HL^q(\lf,M)$ are represented
by cocycles and thus the part of the Leibniz differential constituting the Leibniz coboundary
operator $\dl^q_{\CL^q(\lf,M)}$ is zero. The remaining terms constitute the differential
on $\qf^p\otimes\If$, again since by the Cartan relations for Leibniz cohomology (see
\cite[Proposition~3.1]{LP} for the case of right Leibniz algebras and \cite[Proposition~1.3.2]{C}
for the case of left Leibniz algebras) a Leibniz algebra acts trivially on its cohomology. But
one needs to be careful since the Cartan relations do only hold for $q\ge 1$. Therefore,
for an arbitrary bimodule $M$, $\qf$ will act non-trivially on $\HL^0(\lf,M)$. On the other
hand, if the bimodule $M$ is symmetric, however, the action is indeed trivial on $\HL^0(\lf,M)$. 

Note that in the proof of the preceding lemma, in all action terms on $\If$ or on
$\Hom_\F(\If,\HL^0(\lf,M))$ the action is from the left, thus, in order to interpret the
remaining terms as the Leibniz boundary operator with values in $\If$, we have
to switch around the last action term. This is the reason for viewing $\If$ and
$\Hom_\F(\If,\HL^0(\lf,M))$) here as a symmetric $\qf$-bimodule.

Finally, we obtain from the Universal Coefficient Theorem (for example, see \cite[Theorem
3.6.5]{HA}) that for $q>0$ the second term of the spectral sequence is
$$E_2^{p,q}=\Hom_\F(\HL_p(\qf,\If_s),\HL^q(\lf,M))\,.$$
Furthermore, for a finite-dimensional Leibniz algebra $\lf$ and a symmetric $\qf$-bimodule
$M$, for all $p,q\geq 0$, the $E_1$-term simplifies to 
$$E_1^{p,q}=\CL^p(\qf,\If^*)\otimes\HL^q(\lf,M)\,,$$
and in this special case, for all $p,q\geq 0$, the $E_2$-term is 
$$E_2^{p,q}=\HL^p(\qf,(\If)^*_s)\otimes\HL^q(\lf,M)\,.$$

This discussion proves the following results:

\begin{thm}\label{hs1}
Let $0\to\If\to\lf\stackrel{\pi}\to\qf\to 0$ be a short exact sequence of left Leibniz
algebras such that $\If\subseteq C_\ell(\lf)$. Then $\If$ is in a natural way an
anti-symmetric $\qf$-bimodule via $a\cdot y:=\pi^{-1}(a)y$ and $y\cdot a:=y
\pi^{-1}(a)$ for every element $a\in\qf$ and every element $y\in\If$. Moreover,
there is a spectral sequence converging to $\HL^\bullet(\lf\vert\qf,M)$ with second term
\begin{eqnarray*}
E_2^{p,q}=
\left\{
\begin{array}{cl}
\HL^{p}(\qf,\Hom_\F(\If,\HL^0(\lf,M))_s) & \mbox{if }p\ge 0,q=0\\ 
\Hom_\F(\HL_{p}(\qf,\If_s),\HL^q(\lf,M)) & \mbox{if }p\ge 0,q\ge 1
\end{array}
\right.
\end{eqnarray*}
for every $\qf$-bimodule $M$.
\end{thm}

\begin{cor}\label{pirashvili}
If in Theorem \ref{hs1} the Leibniz algebra $\lf$ is finite dimensional and the $\qf$-bimodule
$M$ is symmetric, then, for any integers $p,q\ge 0$, the $E_2$-term of the spectral
sequence simply reads $$E_2^{p,q}=\HL^{p}(\qf,(\If^*)_s)\otimes\HL^q(\lf,M)\,,$$
where the linear dual $\If^*$ of $\If$ is a left $\lf$-module via $(x\cdot f)(y):=-f(xy)$
for every linear form $f\in\If^*$ and any elements $x\in\lf,y\in\If$. 
\end{cor}

\noindent {\bf Remarks.} 
\begin{itemize}
\item[(a)] According to \cite[Proposition 2.13]{F}, Theorem \ref{hs1} applies to $\If:=\leib(\lf)$
                and $\qf:=\lf_\lie$ (see \cite[Remark 4.2]{P} for the analogous statement for Leibniz
                homology). Note that in the cohomology space $\HL^p(\qf,(\If^*)_s)$, the left
                $\qf$-module $\If^*$ is viewed as a symmetric bimodule while naturally it is 
                an anti-symmetric $\qf$-bimodule. 
\item[(b)] The higher differentials in the spectral sequence are again induced by the Leibniz
                 coboundary operator $\dl^{\bullet}$. Observe that the spectral sequence
                 of Corollary \ref{pirashvili} is isomorphic to the spectral sequence of the cochain bicomplex
                 $\CL^{\bullet}(\qf,(\If^*)_s)\otimes\CL^{\bullet}(\lf,M)$. Therefore the description of
                 the higher differentials
                 can be adapted from \cite{Hue} (see, in particular, Remark 3.2 therein). For example,
                 it is clear, if one of the two differentials in the bicomplex is zero, then all higher
                 differentials vanish. We will see an instance of this case in Example C below.
\item[(c)] One might wonder what one gets when one uses the filtration by the last $p$
                arguments instead of the first $p$ arguments. It turns out that this spectral
                sequence has an $E_2$-term that is more difficult to describe (and which we stated
                erroneously in a first version of this article), because one takes in the $E_2$-term the
                cohomology of a complex which appears as coefficients in the Leibniz cohomology
                that constitutes the $E_1$-term.
\end{itemize}
\vspace{.1cm}

As in the previous section, we illustrate the use of the spectral sequence of Theorem \ref{hs1}
and the associated long exact sequence (see Proposition \ref{les}) by two examples.

In the first example we compute the Leibniz cohomology of the smallest nilpotent non-Lie
left Leibniz algebra with trivial coefficients.
\vspace{.3cm}

\noindent {\bf Example C.} Let $\F$ denote an arbitrary field, and let $\nf:=\F e\oplus
\F f$ be the two-dimensional nilpotent left (and right) Leibniz algebra over $\F$ with
multiplication determined by $ff=e$\,. Then $\leib(\nf)=\F e$, and thus $\nf_\lie$ is a
one-dimensional abelian Lie algebra. Hence $\HL^n(\nf_\lie,\F)\cong\F$ for every
non-negative integer $n$. Moreover, we have that $\dim_\F\HL^0(\nf,\F)=1$ and
$\dim_\F\HL^1(\nf,\F)=1$.

Next, we compute the higher cohomology with the help of the spectral sequence of
Corollary \ref{pirashvili}. As observed in Remark (b) after Corollary~\ref{pirashvili},
all higher differentials are zero in our case since the Leibniz coboundary operator
of the abelian Lie algebra with values in the trivial module vanishes. With the input
data $\dim_\F\HL^0(\nf,\F)=1$ and $\dim_\F\HL^1(\nf,\F)=1$, we therefore get from
the spectral sequence $$\dim_\F\HL^0(\nf|\nf_\lie,\F)=1\,\,\,{\rm and}\,\,\,\dim_\F
\HL^1(\nf|\nf_\lie,\F)=2\,.$$ In order to be able to apply now the long exact sequence
from Proposition \ref{les} and deduce the dimensions of the cohomology spaces, we
want to argue that this sequence is split. In fact, it is split because the connecting
homomorphism is surjective. This comes from the fact that the cochain complex
$\CL^{\bullet}(\nf_\lie,\F)$ is one-dimensional in each degree and a generator can
be hit via the connecting homomorphism which is easy to see directly (take a cochain
in $\CL^n(\nf|\nf_\lie,\F)$ represented by an element in $\CL^{n+1}(\nf,\F)$ with
exactly one slot in $e^*$ at the first place: the Leibniz product in this slot gives the
only non-zero contribution). Consequently, the long exact sequence from Proposition
\ref{les} splits into short exact sequences $$0\to\HL^n(\nf,\F)\to\HL^{n-1}(\nf|\nf_\lie,
\F)\to\HL^{n+1}(\nf_\lie,\F)\to 0\,,$$ starting from $n=2$, where the right-hand term
is one-dimensional. These short exact sequences, together with the spectral sequence
where every differential is zero, permit us to determine all relative and absolute cohomology
spaces. For example, we have $\dim_\F\HL^2(\nf,\F)=1$, and then $\dim_\F\HL^2
(\nf|\nf_\lie,\F)=3$, $\dim_\F\HL^3(\nf,\F)=2$, and then $\dim_\F\HL^3(\nf|\nf_\lie,
\F)=5$, and so on. In general, we obtain by induction that $\dim_\F\HL^n(\nf,\F)=
2^{n-2} $ for every integer $n\ge 2$ and $\dim_\F\HL^n(\nf|\nf_\lie,\F)=2^{n-1}
+1$ for every integer $n\ge 1$.
\vspace{.3cm}

In the second example we compute the Leibniz cohomology of the smallest non-nilpotent
non-Lie left Leibniz algebra with coefficients in one-dimensional bimodules. Note that contrary
to the semidirect product of two one-dimensional Lie algebras in Example A the Leibniz
algebra in Example D is the hemi-semidirect product of two one-dimensional Lie algebras.
It turns out that this somewhat simplifies matters.
\vspace{.3cm}

\noindent {\bf Example D.} Let $\F$ denote an arbitrary field, and let $\Af:=\F h\oplus\F e$
be the two-dimensional supersolvable left Leibniz algebra over $\F$ with multiplication determined
by $he=e$\,. For any scalar $\lambda\in\F$ one can define a one-dimensional left $\Af$-module
$F_\lambda:=\F 1_\lambda$ with $\Af$-action defined by $h\cdot 1_\lambda:=\lambda 1_\lambda$
and $e\cdot 1_\lambda:=0$. Note that $\leib(\Af)=\F e$, and thus $\Af_\lie$ is a one-dimensional
abelian Lie algebra. Then we obtain from \cite[Lemma 1]{B1} and Theorem \ref{vanlie} that
\begin{eqnarray*}
\dim_\F\HL^n(\Af_\lie,(F_\lambda)_s)=
\left\{
\begin{array}{cl}
1 & \mbox{if }\lambda=0\mbox{ and }n\mbox{ is arbitrary}\\ 
0 & \hspace{1cm}\mbox{otherwise}\,.
\end{array}
\right.
\end{eqnarray*}
Moreover, we deduce from Lemma \ref{antisym}\,(b) that
\begin{eqnarray*}
\HL^n(\Af_\lie,(F_\lambda)_a)&\cong&\HL^{n-1}(\Af_\lie,\Hom_\F(\Af_\lie,F_\lambda)_s)\\
&\cong&\HL^{n-1}(\Af_\lie,(F_\lambda)_s)
\end{eqnarray*}
for every integer $n\ge 1$, and therefore
\begin{eqnarray*}
\dim_\F\HL^n(\Af_\lie,(F_\lambda)_a)=
\left\{
\begin{array}{cl}
1 & \mbox{if }\lambda=0\mbox{ and }n\mbox{ is arbitrary or if }\lambda\ne 0\mbox{ and }n=0\\ 
0 & \hspace{3cm}\mbox{otherwise}\,.
\end{array}
\right.
\end{eqnarray*}
%Moreover, it follows from \cite[Corollary 4.2]{F} that $\dim_\F\HL^0(\Af,\F)=1$ and from
%\cite[Corollary 4.4\,(a)]{F} that $\dim_\F\HL^1(\Af,\F)=1$.

In order to be able to apply the spectral sequence of Theorem \ref{hs1}, we first compute
$\HL^\bullet(\Af_\lie,[\leib(\Af)^*]_s)$. Observe that the module $\leib(\Af)^*=\F e^*\cong
F_{-1}$ is non-trivial irreducible, and furthermore it is viewed as a symmetric $\Af_\lie$-bimodule.
Hence from the above it follows that $\HL^n(\Af_\lie,[\leib(\Af)^*]_s)=0$ for every
non-negative integer $n$. This implies in turn that the spectral sequence of Theorem \ref{hs1}
collapses at the $E_2$-term and that 
\begin{eqnarray*}
\HL^n(\Af\vert\Af_\lie,(F_\lambda)_a)&=&\HL^n(\Af_\lie,\Hom_\F(\leib(\Af),\HL^0(\Af,
(F_\lambda)_a))_s)\\
&=&\HL^n(\Af_\lie,\Hom_\F(\leib(\Af),F_\lambda)_s)
\end{eqnarray*}
for all non-negative integers $n$, while $\HL^n(\Af\vert\Af_\lie,(F_\lambda)_s)=0$ for all
$n\ge 0$ by Corollary \ref{pirashvili}. Notice that as an $\Af$-bimodule $\Hom_\F(\leib(\Af),
F_\lambda)_s\cong[F_{\lambda-1}]_s$. We have already observed in Example C that the
long exact sequence of Proposition~\ref{les} splits, and therefore we conclude from Proposition
\ref{les} that
$$\HL^n(\Af,(F_{\lambda})_s)\cong\HL^n(\Af_\lie,(F_{\lambda})_s)$$
and 
$$\HL^n(\Af,(F_{\lambda})_a)\cong\HL^n(\Af_\lie,(F_{\lambda})_a)\oplus\HL^n(\Af_\lie,
(F_{\lambda-1})_s)$$
for all $\lambda$ and all non-negative integers $n$. Consequently, we obtain that
\begin{eqnarray*}
\dim_\F\HL^n(\Af,(F_\lambda)_s)=
\left\{
\begin{array}{cl}
1 & \mbox{if }\lambda=0\mbox{ and }n\mbox{ is arbitrary}\\ 
0 & \hspace{1cm}\mbox{otherwise}\,,
\end{array}
\right.
\end{eqnarray*}
and
\begin{eqnarray*}
\dim_\F\HL^n(\Af,(F_\lambda)_a)=
\left\{
\begin{array}{cl}
1 & \mbox{if }\lambda=0,1\mbox{ and }n\mbox{ is arbitrary or if }\lambda\ne 0,1\mbox{ and }n=0\\ 
0 & \hspace{3cm}\mbox{otherwise}\,.
\end{array}
\right.
\end{eqnarray*}
\vspace{.3cm}

\noindent {\bf Remark.}
In particular, we have that $\dim_\F\HL^n(\Af,\F)=1$ for every non-negative integer $n$. Note
that this follows as well from the scheme of proof of Proposition 4.3 in \cite{P} by using the
isomorphism between Leibniz homology and cohomology with trivial coefficients. Indeed, the
characteristic element $\chelem(\Af)\in\HL^2(\Af_\lie,\leib(\Af))$ of $\Af$ is zero as $\leib(\Af)
=\F e\cong(F_1)_a$ and $\HL^2(\Af_\lie,(F_1)_a)=0$. Since also $\HL^\bullet(\Af_\lie,[{\rm Leib}
(\Af)^*]_s)$ is zero, we can reason in the same way as Pirashvili does.
\vspace{.2cm}

%%%%%%%%%%%%%%%%%%%%%%%%%%%%%%%%%%%%%%%%%%%%%%%%%%

\section{Cohomology of semi-simple Leibniz algebras}\label{semsim}

%%%%%%%%%%%%%%%%%%%%%%%%%%%%%%%%%%%%%%%%%%%%%%%%%%

Recall that a left Leibniz algebra $\lf$ is called {\em semi-simple\/} if $\leib(\lf)$ contains
every solvable ideal of $\lf$ (see \cite[Section~7]{F}). In particular, a finite-dimensional
left Leibniz algebra $\lf$ is semi-simple if, and only if, $\leib(\lf)=\rad(\lf)$, where $\rad
(\lf)$ denotes the largest solvable ideal of $\lf$ (see \cite[Proposition 7.4]{F}). Moreover,
a left Leibniz algebra $\lf$ is semi-simple if, and only if, the canonical Lie algebra $\lf_\lie$
associated to $\lf$ is semi-simple (see \cite[Proposition 7.8]{F}).

\begin{lem}\label{leibnizinv}
Let $\lf$ be a finite-dimensional semi-simple left Leibniz algebra over a field of characteristic
zero. Then $[\leib(\lf)^*]_s^{\lf_\lie}=0$, where $\leib(\lf)^*$ is a left $\lf$-module, and
thus a left $\lf_\lie$-module, via $(x\cdot f)(y):=-f(xy)$ for every linear form $f\in\leib(\lf)^*$
and any elements $x,y\in\lf$.
\end{lem}

\begin{proof}
It follows from Levi's theorem for Leibniz algebras (see \cite[Proposition 2.4]{P} or
\cite[Theorem 1]{B2}) that there exists a semi-simple Lie subalgebra $\ssf$ of $\lf$
such that $\lf=\ssf\oplus\leib(\lf)$ (see also \cite[Corollary 2.14]{FM}). Note that then
$\lf_\lie\cong\ssf$. Since $\ssf$ is a Lie algebra and $\leib(\lf)\subseteq C_\ell(\lf)$,
we obtain that $(s+x)(s+x)=sx$ for any elements $s\in\ssf$ and $x\in\leib(\lf)$. This
shows that $\leib(\lf)=\ssf\leib(\lf)$. Now let $\varphi\in[\leib(\lf)^*]_s^\ssf$ be
arbitrary. Since $(s\cdot\varphi)(x)=-\varphi(sx)$ for any $\varphi\in\leib(\lf)^*$,
$s\in\ssf$, and $x\in\leib(\lf)$, we conclude that $\varphi[\leib(\lf)]=\varphi[\ssf
\leib(\lf)]=0$, which proves the assertion.
\end{proof}

The first main  result in this section is the Leibniz analogue of Whitehead's vanishing
theorem for the Chevalley-Eilenberg cohomology of finite-dimensional semi-simple
Lie algebras over a field of characteristic zero (see \cite[Theorem 24.1]{CE} or
\cite[Theorem 10]{HS}). Note that in the special case of a Lie algebra,
Theorem~\ref{whitehead} is an immediate consequence of Whitehead's classical
vanishing theorem and Theorem~\ref{vanlie}.

\begin{thm}\label{whitehead}
Let $\lf$ be a finite-dimensional semi-simple left Leibniz algebra over a field of
characteristic zero. If $M$ is a finite-dimensional $\lf$-bimodule such that $M^\lf
=0$, then $\HL^n(\lf,M)=0$ for every non-negative integer $n$.
\end{thm}

\begin{proof}
According to Lemma \ref{sym}, the hypothesis $M^\lf=0$ implies that $M$ is symmetric. 
We can therefore use the spectral sequence of Corollary \ref{pirashvili} with $\If:=\leib
(\lf)$ and $\qf:=\lf_\lie$. The $E_2$-term reads $$E_2^{p,q}=\HL^p(\qf,(\If^*)_s)
\otimes\HL^q(\lf,M)\,.$$ It follows from \cite[Proposition 7.8]{F} and the Ntolo-Pirashvili
vanishing theorem for the Leibniz cohomology of a finite-dimensional semi-simple Lie algebra
over a field of characteristic zero (see \cite[Th\'eor\`eme 2.6]{N} and the sentence after
the proof of Lemma 2.2 in \cite{P}) that $\HL^p(\qf,(\If^*)_s)=0$ for every positive
integer $p$. Hence the spectral sequence collapses, and we deduce $$\HL^n(\lf|\qf,M)
=(\If^*)_s^\qf\otimes\HL^n(\lf,M)\,.$$ By virtue of Lemma \ref{leibnizinv}, the relative
cohomology $\HL^n(\lf\vert\qf,M)$ vanishes for every non-negative integer $n$, and
thus we obtain from Proposition \ref{les} in conjunction with \cite[Proposition 4.1]{F}
and the Ntolo-Pirashvili vanishing theorem that $\HL^n(\lf,M)\cong\HL^n(\qf,M)=0$
for every non-negative integer $n$.
\end{proof}

\noindent {\bf Remark.} It is possible to prove Theorem \ref{whitehead} without using
the Ntolo-Pirashvili vanishing theorem. Namely, the first time the Ntolo-Pirashvili vanishing
theorem is used in the above proof, one can instead use Lemma \ref{leibnizinv}, Whitehead's
classical vanishing theorem, and Theorem \ref{vanlie}, and the second time, by hypothesis,
it is enough to apply just the last two results. As a consequence, the proof of
Theorem~\ref{vansemsim} gives also a new proof of the Ntolo-Pirashvili vanishing theorem.
\vspace{.3cm}

Next, we generalize the Ntolo-Pirashvili vanishing theorem from Lie algebras to arbitrary
Leibniz algebras. The main tools in the proof are Corollary \ref{irr}, Theorem~\ref{whitehead},
Corollary \ref{coadj}, and Lemma \ref{antisym}, where the second result and its use in this
proof seems to be new.

\begin{thm}\label{vansemsim}
Let $\lf$ be a finite-dimensional semi-simple left Leibniz algebra over a field of characteristic
zero, and let $M$ be a finite-dimensional $\lf$-bimodule. Then $\HL^n(\lf,M)=0$ for
every integer $n\ge 2$, and there is a five-term exact sequence $$0\to M_0\to\HL^0(\lf,M)
\to M_\sym^{\lf_\lie}\to\Hom_\lf(\lf_{\ad,\ell},M_0)\to\HL^1(\lf,M)\to 0\,.$$ Moreover, if $M$
is symmetric, then $\HL^n(\lf,M)=0$ for every integer $n\ge 1$.
\end{thm}

\begin{proof}
The proof is divided into three steps. First, we will prove the assertion for symmetric
$\lf$-bimodules. So suppose that $M$ is symmetric. Since $M$ is finite-dimensional, it
has a composition series. It is clear that sub-bimodules and homomorphic images of a
symmetric bimodule are again symmetric. By using the long exact cohomology sequence,
it is therefore enough to prove the second part of the theorem for finite-dimensional
irreducible symmetric $\lf$-bimodules. So suppose now in addition that $M$ is irreducible
and non-trivial. Then we obtain from Corollary \ref{irr} that $M^\lf=0$, and thus
Theorem \ref{whitehead} yields that $\HL^n(\lf,M)=0$ for every non-negative integer
$n$. Finally, suppose that $M=\F$ is the trivial irreducible $\lf$-bimodule. In this case
it follows from Corollary \ref{coadj} that $\HL^n(\lf,\F)\cong\HL^{n-1}(\lf,(\lf^*)_s)$
for every integer $n\ge 1$. Since $\lf_\lie$ is perfect, we obtain from Corollary \ref{coadj}
that $$(\lf^*)_s^\lf\cong\HL^0(\lf,(\lf^*)_s)\cong\HL^1(\lf,\F)\cong\HCE^1(\lf_\lie,
\F)=0\,.$$ Therefore another application of Theorem \ref{whitehead} yields that
$$\HL^n(\lf,\F)\cong\HL^{n-1}(\lf,(\lf^*)_s)=0$$ for every integer $n\ge 1$. This
finishes the proof for symmetric $\lf$-bimodules.

If $M$ is anti-symmetric, then we obtain the assertion from Lemma \ref{antisym}\,(b)
and the statement for symmetric bimodules. Finally, if $M$ is arbitrary, then in the short
exact sequence $0\to M_0\to M\to M_\sym\to 0$ the first term is anti-symmetric and the
third term is symmetric. Hence another application of the long exact cohomology sequence
in conjunction with the statement for the anti-symmetric and the symmetric case yields
$\HL^n(\lf,M)=0$ for every integer $n\ge 2$. Now we deduce the five-term exact
sequence from the long exact cohomology sequence together with \cite[Corollary 4.2]{F},
\cite[Corollary 4.4\,(b)]{F}, and the symmetric case.
\end{proof}

Note that Theorem \ref{vansemsim} contains \cite[Theorem 7.15]{F} as the special
case $n=1$ and the second Whitehead lemma for Leibniz algebras as the special case
$n=2$. But contrary to Chevalley-Eilenberg cohomology, Leibniz cohomology vanishes
in any degree $n\ge 2$.

The following example shows that the Ntolo-Pirashvili vanishing theorem (and therefore
also Theorem \ref{vansemsim}) does not hold over fields of prime characteristic.
\vspace{.3cm}

\noindent {\bf Example E.} Let $\gf:=\slf_2(\F)$ be the three-dimensional simple Lie algebra of
traceless $2\times 2$ matrices over a field $\F$ of characteristic $p>2$. Moreover, let $\F_p$
denote the field with $p$ elements, and let $L(n)$ ($n\in\F_p$) denote the irreducible restricted
$\gf$-module of heighest weight $n$. (If the ground field $\F$ is algebraically closed, these
modules represent all isomorphism classes of restricted irreducible $\gf$-modules.) It is well
known (see \cite[Theorem 4]{Dz}) that $\HCE^1(\gf,L(p-2))\cong\F^2\cong\HCE^2(\gf,L(p-2))$.
(Note that by virtue of \cite[Theorem 2]{Dz}, $\HCE^\bullet(\gf,M)=0$ for every non-restricted
irreducible $\gf$-module. In fact, $L(p-2)$ is the only irreducible $\gf$-module $M$ such that
$\HCE^1(\gf,M)\ne 0$ or $\HCE^2(\gf,M)\ne 0$.)

We obtain from Proposition \ref{lesrelcoh} that $$\HL^1(\gf,L(p-2)_s)\cong\HCE^1(\gf,L(p-2))
\cong\F^2\ne 0$$ and $$0\ne\F^2\cong\HCE^2(\gf,L(p-2))\hookrightarrow\HL^2(\gf,L(p-2)_s)\,.$$
In particular, this shows that the Ntolo-Pirashvili vanishing theorem (and therefore also Theorem
\ref{vansemsim}) is not true over fields of prime characteristic.
\vspace{.3cm}

\noindent {\bf Remark.} By using more sophisticated tools one can also say something about
the Leibniz cohomology of anti-symmetric irreducible $\gf$-bimodules, where again $\gf:=
\slf_2(\F)$. We obtain from Lemma \ref{antisym}\,(b) that $$\HL^1(\gf,L(n)_a)\cong\HL^0
(\gf,\Hom_\F(\gf,L(n))_s)\cong\Hom_\F(\gf,L(n))^\gf$$ and $$\HL^2(\gf,L(n)_a)\cong\HL^1
(\gf,\Hom_\F(\gf,L(n))_s)\cong\HCE^1(\gf,\Hom_\F(\gf,L(n)))\,.$$ Since $\gf\cong L(2)$ is
a self-dual $\gf$-module, we have the following isomorphisms of $\gf$-modules: $$\Hom_\F
(\gf,L(n))\cong L(2)\otimes L(n)\,.$$ Let us first consider the case $p>3$. Then we obtain
from the modular Clebsch-Gordan rule (see \cite[Theorem 1.11\,(a)]{BO} or Satz a) in Chapter
5 of \cite{Fi}) that $$L(2)\otimes L(2)\cong L(4)\oplus L(2)\oplus L(0)$$ and 
\begin{eqnarray*}
L(2)\otimes L(p-4)\cong
\left\{
\begin{array}{cl}
L(3)\oplus L(1) & \mbox{\rm if }p=5\\ 
L(p-2)\oplus L(p-4)\oplus L(p-6) & \mbox{\rm if }p\ge 7\,.
\end{array}
\right.
\end{eqnarray*}
Hence we conclude for $p>3$ that $$\HL^1(\gf,L(2)_a)\cong(L(2)\otimes L(2))^\gf\cong L(0)^\gf
\cong\F\ne 0$$ and $$\HL^2(\gf,L(p-4)_a)\cong\HCE^1(\gf,L(2)\otimes L(p-4))\cong\HCE^1(\gf,
L(p-2))\cong\F^2\ne 0\,.$$

Let us now consider $p=3$. Note that in this case $L(2)$ is the {\em Steinberg module\/}, i.e.,
$L(2)$ is the unique projective irreducible restricted $\gf$-module. This implies that $L(2)\otimes
L(n)$ is also projective for every highest weight $n\in\F_3$. Then we obtain from the modular
Clebsch-Gordan rule (cf.\ \cite[Theorem 1.11\,(b) and (c)]{BO} or Satz b) and c) in Chapter 5
of \cite{Fi}) for $p=3$ that
\begin{eqnarray*}
L(2)\otimes L(n)\cong
\left\{
\begin{array}{cl}
L(2) & \mbox{\rm if }n\equiv 0\,\,(\mathrm{mod}\,3)\\
P(1) & \mbox{\rm if }n\equiv 1\,\,(\mathrm{mod}\,3)\,,\\ 
P(0)\oplus L(2) & \mbox{\rm if }n\equiv 2\,\,(\mathrm{mod}\,3)
\end{array}
\right.
\end{eqnarray*}
where $P(n)$ denotes the projective cover (and at the same time also the injective hull) of $L(n)$.
As a consequence, we have that
\begin{eqnarray*}
(L(2)\otimes L(n))^\gf\cong
\left\{
\begin{array}{cl}
\F & \mbox{\rm if }n\equiv 2\,\,(\mathrm{mod}\,3)\\
0 & \mbox{\rm if }n\not\equiv 2\,\,(\mathrm{mod}\,3)\,.
\end{array}
\right.
\end{eqnarray*}
Therefore, we obtain that $$\HL^1(\gf,L(2)_a)\cong(L(2)\otimes L(2))^\gf\cong\F\ne 0\,.$$
Moreover, by using the six-tem exact sequence relating Hochschild's cohomology of a
restricted Lie algebra to its Chevalley-Eilenberg cohomology (see \cite[p.\ 575]{Ho}),
we also conclude that $$\HL^2(\gf,L(2)_a)\cong\HCE^1(\gf,L(2)\otimes L(2))\cong
\gf^*\cong\F^3\ne 0\,.$$
\vspace{-.3cm}

The next example shows that the Ntolo-Pirashvili vanishing theorem (and therefore also
Theorem \ref{vansemsim}) does not hold for infinite-dimensional modules.
\vspace{.3cm}

\noindent {\bf Example F.} Let $\gf:=\slf_2(\C)$ be the three-dimensional simple complex Lie algebra
of traceless $2\times 2$ matrices, and let $V(\lambda)$ ($\lambda\in\C$) denote the {\em Verma
module\/} of highest weight $\lambda$. (Here we identify every complex multiple of the unique
fundamental weight with its coefficient.) Verma modules are infinite-dimensional indecomposable
$\gf$-modules (see, for example, \cite[Theorem 20.2\,(e)]{H}). Furthermore, it is well known (see
\cite[Exercise 7\,(b) \& (c) in Section 7.2]{H}) that $V(\lambda)$ is irreducible if, and only if, $\lambda$
is not a dominant integral weight (i.e., with our identification, $\lambda$ is not a non-negative
integer). Moreover, it follows from \cite[Theorem 4.19]{W} that
\begin{eqnarray*}
\HCE^n(\gf,V(\lambda))\cong
\left\{
\begin{array}{cl}
\C & \mbox{\rm if }\lambda=-2\mbox{ and }n=1,2\\ 
0 & \hspace{1cm}\mbox{\rm otherwise}\,.
\end{array}
\right.
\end{eqnarray*}
This in conjunction with Proposition \ref{lesrelcoh} yields that $$\HL^1(\gf,V(-2)_s)\cong
\HCE^1(\gf,V(-2))\cong\C\ne 0$$ and $$0\ne\C\cong\HCE^2(\gf,V(-2))\hookrightarrow
\HL^2(\gf,V(-2)_s)\,.$$ In particular, the Ntolo-Pirashvili vanishing theorem (and therefore
also Theorem~\ref{vansemsim}) is not true for infinite-dimensional modules.
\vspace{.3cm}

We obtain as an immediate consequence of Theorem \ref{vansemsim} the following generalization
of \cite[Corollary 7.9]{F}.

\begin{cor}\label{triv}
If $\lf$ is a finite-dimensional semi-simple left Leibniz algebra over a field of characteristic zero,
then $\HL^n(\lf,\F)=0$ for every integer $n\ge 1$.
\end{cor}

\noindent {\bf Remark.} It is well known that the analogue of Corollary \ref{triv} does not
hold for the Chevalley-Eilenberg cohomology of Lie algebras as $\HCE^3(\gf,\F)\ne 0$ for
any finite-dimensional semi-simple Lie algebra $\gf$ over a field $\F$ of characteristic zero
(see \cite[Theorem 21.1]{CE}).
\vspace{.3cm}

Next, we apply Theorem \ref{vansemsim} to compute the cohomology of a finite-dimensional
semi-simple left Leibniz algebra over a field of characteristic zero with coefficients in its
adjoint bimodule and in its (anti-)symmetric counterparts.

\begin{thm}\label{adj}
For every finite-dimensional semi-simple left Leibniz algebra $\lf$ over a field of characteristic
zero the following statements hold:
\begin{enumerate}
\item[{\rm (a)}]
\begin{eqnarray*}
\HL^n(\lf,\lf_s)=
\left\{
\begin{array}{cl}
\leib(\lf) & \mbox{\rm if }n=0\\ 
0 & \mbox{\rm if }n\ge 1\,.
\end{array}
\right.
\end{eqnarray*}
\item[{\rm (b)}]
\begin{eqnarray*}
\HL^n(\lf,\lf_a)=
\left\{
\begin{array}{cl}
\lf & \mbox{\rm if }n=0\\
\End_\lf(\lf_{\ad,\ell}) & \mbox{\rm if }n=1\,,\\ 
0 & \mbox{\rm if }n\ge 2
\end{array}
\right.
\end{eqnarray*}
where $\End_\lf(\lf_{\ad,\ell})$ denotes the vector space of endomorphisms of the left adjoint
$\lf$-module $\lf_{\ad,\ell}$.
\item[{\rm (c)}]
\begin{eqnarray*}
\HL^n(\lf,\lf_\ad)=
\left\{
\begin{array}{cl}
\leib(\lf) & \mbox{\rm if }n=0\\
\Hom_\lf(\lf_{\ad,\ell},\leib(\lf)) & \mbox{\rm if }n=1\,,\\ 
0 & \mbox{\rm if }n\ge 2
\end{array}
\right.
\end{eqnarray*}
where $\Hom_\lf(\lf_{\ad,\ell},\leib(\lf))$ denotes the vector space of homomorphisms from
the left adjoint $\lf$-module $\lf_{\ad,\ell}$ to the Leibniz kernel $\leib(\lf)$ considered as a
left $\lf$-module.
\end{enumerate}
\end{thm}

\begin{proof}
(a): According to \cite[Proposition 4.1]{F} and \cite[Proposition 7.5]{F} we have that that
$\HL^0(\lf,\lf_s)=(\lf_s)^\lf=C_\ell(\lf)=\leib(\lf)$. Moreover, we obtain the statement for
degree $n\ge 1$ from the second part of Theorem \ref{vansemsim}.

(b): It follows from \cite[Corollary 4.2\,(b)]{F} that $\HL^0(\lf,\lf_a)=\lf$, and it follows from
\cite[Corollary 4.4\,(b)]{F} that $\HL^1(\lf,\lf_a)=\End_\lf(\lf_{\ad,\ell})$. The remainder of
the assertion is an immediate consequence of the first part of Theorem \ref{vansemsim}.

(c): As for the symmetric adjoint bimodule, we obtain from \cite[Proposition 4.1]{F} and
\cite[Proposition 7.5]{F} that $\HL^0(\lf,\lf_\ad)=(\lf_\ad)^\lf=C_\ell(\lf)=\leib(\lf)$.
Next, by applying the five-term exact sequence of Theorem \ref{vansemsim} to the adjoint
$\lf$-bimodule $M:=\lf_\ad$, we deduce that $$\HL^1(\lf,\lf_\ad)\cong\Hom_\lf(\lf_{\ad,
\ell},\leib(\lf))\,,$$ as the third term is $\lf_\lie^{\lf_\lie}=C(\lf_\lie)=0$. Finally, the
assertion for degree $n\ge 2$ is again an immediate consequence of the first part of
Theorem \ref{vansemsim}.
\end{proof}

\noindent {\bf Remark.} Note that the vanishing part of Theorem \ref{adj}\,(c) confirms
a generalization of the conjecture at the end of \cite{ALO}. Moreover, parts (a) and (b)
of Theorem \ref{adj} show that the statements in Theorem \ref{vansemsim} are best
possible.
\vspace{.3cm}

In particular, one can derive from Theorem \ref{adj}\,(c) that finite-dimensional semi-simple
non-Lie Leibniz algebras over a field of characteristic zero have outer derivations. In this
respect non-Lie Leibniz algebras behave differently than Lie algebras (see, for example,
\cite[Theorem 5.3]{H}).

\begin{cor}\label{outder}
Every finite-dimensional semi-simple non-Lie Leibniz algebra over a field of characteristic zero
has derivations that are not inner.
\end{cor}

\begin{proof}
If one applies the contravariant functor $\Hom_\F(-,\leib(\lf))$ to the short exact sequence
$$0\to\leib(\lf)\to\lf_\ad\to\lf_\lie\to 0$$ considered as a short exact sequence of left
$\lf$-modules, one obtains the short exact sequence $$0\to\Hom_\F(\lf_\lie,\leib(\lf))\to
\Hom_\F(\lf_{\ad,\ell},\leib(\lf))\to\Hom_\F(\leib(\lf),\leib(\lf))\to 0$$ of left $\lf$-modules.
Then the long exact cohomology sequence in conjunction with Lemma \ref{antisym}\,(a)
yields the long exact sequence
\begin{eqnarray*}
0 & \to & \Hom_\lf(\lf_\lie,\leib(\lf))\to\Hom_\lf(\lf_{\ad,\ell},\leib(\lf))\to\Hom_\lf(\leib(\lf),\leib(\lf))\\
& \to & \hl^1(\lf,\Hom_\F(\lf_\lie,\leib(\lf)))=\HL^1(\lf,\Hom_\F(\lf_\lie,\leib(\lf))_s)\,.
\end{eqnarray*}
According to the second part of Theorem \ref{vansemsim}, the fourth term is zero. Since the third
term contains the identity map, it is non-zero as by hypothesis $\lf$ is a not a Lie algebra. Hence in
this case the second term is non-zero, and we obtain from Theorem \ref{adj}\,(c) that $\HL^1(\lf,
\lf_\ad)\cong\Hom_\lf(\lf_{\ad,\ell},\leib(\lf))\ne 0$.
\end{proof}

\noindent {\bf Remark.} After the submission of our paper we became aware of the preprint
\cite{BMS} in which the authors introduce a more general concept of inner derivations for
Leibniz algebras than in our paper. Namely, a derivation $D$ of a left Leibniz algebra $\lf$ is
called {\em inner\/} if there exists an element $x\in\lf$ such that $\im(D-L_x)\subseteq
\leib(\lf)$. Then it is shown that every derivation of a finite-dimensional semi-simple Lie
algebra over a field of characteristic zero is inner in this more general sense (see
\cite[Theorem~3.3]{BMS}).
\vspace{-.1cm}

In the same way as at the end of Section \ref{cel} for the infinite-dimensional two-sided Witt
algebra, by using \cite[Th\'eor\`eme 3]{Ba} in conjunction with Theorem~\ref{adj}\,(c), one
obtains the rigidity of any finite-dimensional semi-simple Lie algebra as a Leibniz algebra.

\begin{cor}\label{rigid}
Every finite-dimensional semi-simple left Leibniz algebra over an algebraically closed field of
characteristic zero is rigid as a Leibniz algebra.
\end{cor}
\vspace{.3cm}

%%%%%%%%%%%%%%%%%%%%%%%%%%%%%%%%%%%%%%%%%%%%%%%%%%

\noindent {\bf Acknowledgments.} Most of this paper was
written during a sabbatical leave of the first author in the Fall semester 2018.
He is very grateful to the University of South Alabama for giving him this
opportunity.

The first author would also like to thank the Laboratoire de math\'ematiques
Jean Leray at the Universit\'e de Nantes for the hospitality and the financial
support in the framework of the program D\'efiMaths during his visit in August
and September 2018. Moreover, the first author wishes to thank Henning Krause
and the BIREP group at Bielefeld University for the hospitality and the financial
support during his visit in October and November 2018 when large portions of
the paper were written.

Both authors would like to thank Bakhrom Omirov for useful discussions. We are
also grateful to Teimuraz Pirashvili for spotting a mistake in a previous version of
the manuscript and to Geoffrey Powell for his help in understanding this mistake
as well as for several useful remarks that improved the present manuscript. Finally,
we would like to thank an anonymous referee for correcting a mistake in the submitted
version. 
\vspace{.2cm}

%%%%%%%%%%%%%%%%%%%%%%%%%%%%%%%%%%%%%%%%%%%%%%%%%%

%%%%%%%%%%%%%%%%%%%%%%%%%%%%%%%%%%%%%%%%%%%%%%%%%%

\end{document}